\documentclass[12pt,a4paper]{article}
\usepackage[top=3cm, bottom=3cm, left=3cm, right=3cm]{geometry}
\usepackage{breqn}
\usepackage[hang,flushmargin]{footmisc} %nonindent footnote

\usepackage{amsmath,amsthm,amssymb,graphicx, multicol, array}
\usepackage{enumerate}
\usepackage{enumitem}
\usepackage{setspace}
\usepackage{xcolor}
\usepackage{currfile}
\usepackage{float}

\usepackage[pagebackref,bookmarks,colorlinks,breaklinks]{hyperref}
\hypersetup{linkcolor=blue,citecolor=red,filecolor=blue,urlcolor=blue} 

\usepackage{tabto}

\DeclareMathOperator{\tmod}{mod}

\newtheorem{thm}{Theorem}[section]
\newtheorem{lem}{Lemma}[section]
\newtheorem{prop}{Proposition}[section]
\newtheorem{conj}{Conjecture}[section]
\newtheorem{exa}{Example}[section]
\newtheorem{cor}{Corollary}[section]
\newtheorem{dfn}{Definition}[section]
\newtheorem{exe}{Exercise}[section]

\newcommand{\N}{\mathbb{N}}
\newcommand{\Z}{\mathbb{Z}}
\newcommand{\Q}{\mathbb{Q}}

\newcommand{\C}{\mathbb{C}}
\newcommand{\F}{\mathbb{F}}
\newcommand{\tP}{\mathbb{P}}

\usepackage{fancyhdr}
\title{Small Primes And The Tau Function}
\date{}
\author{N. A. Carella}

\begin{document}
%\doublespacing
\thispagestyle{empty}
\date{}
%\linenumbers
\maketitle

\textbf{\textit{Abstract}:} This note shows that the prime values of the Ramanujan tau function $\tau(n)=\pm p$ misses every prime $p\leq 8.0\times 10^{25}$. 
\let\thefootnote\relax\footnote{ \today \date{} \\
\textit{AMS MSC}: Primary 11F35; Secondary 11D45 \\
\textit{Keywords}: Tau Function; Modular Function; Lehmer Conjecture.}
\tableofcontents
%\vskip .25 in 
\newpage

%SSSSSSSSSSSSSSSSSSSSSSSSSSSSSSSSSSSSSSSSSSSSSSSSSSSSSSSSSSSSSSSSSSSSSSSSSSSSSSSSSSS
%SSSSSSSSSSSSSSSSSSSSSSSSSSSSSSSSSSSSSSSSSSSSSSSSSSSSSSSSSSSSSSSSSSSSSSSSSSSSSSSSSSS
%SSSSSSSSSSSSSSSSSSSSSSSSSSSSSSSSSSSSSSSSSSSSSSSSSSSSSSSSSSSSSSSSSSSSSSSSSSSSSSSSSSS
%SSSSSSSSSSSSSSSSSSSSSSSSSSSSSSSSSSSSSSSSSSSSSSSSSSSSSSSSSSSSSSSSSSSSSSSSSSSSSSSSSSS
%SSSSSSSSSSSSSSSSSSSSSSSSSSSSSSSSSSSSSSSSSSSSSSSSSSSSSSSSSSSSSSSSSSSSSSSSSSSSSSSSSSS
%SSSSSSSSSSSSSSSSSSSSSSSSSSSSSSSSSSSSSSSSSSSSSSSSSSSSSSSSSSSSSSSSSSSSSSSSSSSSSSSSSSS
%SSSSSSSSSSSSSSSSSSSSSSSSSSSSSSSSSSSSSSSSSSSSSSSSSSSSSSSSSSSSSSSSSSSSSSSSSSSSSSSSSSS
\section{Introduction } \label{s9099}\hypertarget{s9099}
The Fourier coefficient of the discriminant function
\begin{equation} \label{eq9999.003}
\Delta(z)=\sum_{n \geq1} \tau(n)q^n=q-24q^2+252q^3-1472q^4+\cdots,
\end{equation}
where $z\in \mathbb{C}$ is a complex number in the upper half plane, $q=e^{i2\pi }$, is known as the Ramanujan tau function. The classification as a level $N=1$ and weight $k=12$ , and other information appears in \cite{LMFDB}. The prime values of the tau function $\tau: \mathbb{N} \longrightarrow \mathbb{Z}$ is a topic of interest in many papers. The first prime value
\begin{equation} \label{eq9999.005}
\tau\left (251^2\right )=-80561663527802406257321747
\end{equation}
was discovered by Lehmer in \cite{LD1965}. About half a century later, many other prime values have been discovered by many authors. The authors in \cite{LR2013} extended the numerical works of Lehmer and other authors to the range of integers $n\leq 10^{600}$, and beyond. But it remains unknown if $\tau(n)$ has small prime values for some extremely large integers $n\geq1$. A very recent paper has proved that 
\begin{equation} \label{eq9999.023}
\tau\left (n\right )\ne \pm3, \pm5, \pm7,  \pm691,
\end{equation}
see {\color{red}\cite[Theorem 1.1]{BO2020}}, and a slightly larger range of primes $|q|\leq 100$ in {\color{red}\cite[Theorem 6]{BS2021}}. This note offers a simpler proof, and extends the range of primes missed by the Ramanujan tau function.

\begin{thm} \label{thm9999.001}\hypertarget{thm9999.001} Let $q< \tau(251^2)$ be a prime, then $\tau\left (p^{2n}\right )\ne  \pm q$. 
\end{thm}
The underlining foundation of this result are presented in \hyperlink{S7799}{Section} \ref{S7799} to \hyperlink{S9922}{Section} \ref{S9922} and the proof of \hyperlink{thm9999.001}{Theorem} \ref{thm9999.001} appears in \hyperlink{S9099SPV-R}{Section} \ref{S9099SPV-R}.

%SSSSSSSSSSSSSSSSSSSSSSSSSSSSSSSSSSSSSSSSSSSSSSSSSSSSSSSSSSSSSSSSSSSSSSSSSSSSSSSSSSS
%SSSSSSSSSSSSSSSSSSSSSSSSSSSSSSSSSSSSSSSSSSSSSSSSSSSSSSSSSSSSSSSSSSSSSSSSSSSSSSSSSSS
%SSSSSSSSSSSSSSSSSSSSSSSSSSSSSSSSSSSSSSSSSSSSSSSSSSSSSSSSSSSSSSSSSSSSSSSSSSSSSSSSSSS
%SSSSSSSSSSSSSSSSSSSSSSSSSSSSSSSSSSSSSSSSSSSSSSSSSSSSSSSSSSSSSSSSSSSSSSSSSSSSSSSSSSS
%SSSSSSSSSSSSSSSSSSSSSSSSSSSSSSSSSSSSSSSSSSSSSSSSSSSSSSSSSSSSSSSSSSSSSSSSSSSSSSSSSSS
%SSSSSSSSSSSSSSSSSSSSSSSSSSSSSSSSSSSSSSSSSSSSSSSSSSSSSSSSSSSSSSSSSSSSSSSSSSSSSSSSSSS
%SSSSSSSSSSSSSSSSSSSSSSSSSSSSSSSSSSSSSSSSSSSSSSSSSSSSSSSSSSSSSSSSSSSSSSSSSSSSSSSSSSS
\section{Basic Diophantine Approximations Results} \label{S2002}
\begin{thm}{\normalfont ({\color{red}\cite[Theorem 7.8]{NI1956}})} \label{thm2002.500}\hypertarget{thm2002.500} A real algebraic number $\alpha>0$ of degree $n\geq 1$ is not approximable to order $n+1$ or higher. In addition, there is a constant $A\geq 1$ such that
\begin{enumerate} [font=\normalfont, label=(\roman*)]
\item $ \displaystyle 
\left |\alpha-\frac{p_m}{q_m}\right |>\frac{1}{A q_m^n},
$
\item $ \displaystyle 
\left |\alpha-\frac{p_m}{q_m}\right |\leq\frac{1}{A q_m^{n+1}},$
\end{enumerate}
as $q_m \to \infty.$
\end{thm}
\begin{proof}[\textbf{Proof}]  The proof is widely available in the literature, including the cited reference.
\end{proof}

\begin{cor} \label{cor2002.500}\hypertarget{cor2002.500}Let $\beta$ be an algebraic number such that $\beta\approx 1$, and let $\{p_m/q_m:m\geq1\}$ be the sequence of convergents. Then,
\begin{enumerate} [font=\normalfont, label=(\roman*)]
\item $ \displaystyle 
\left |\beta-\frac{p_m}{q_m}\right |>\frac{1}{n^32^n q_m^n},
$
\item $ \displaystyle 
\left |\beta-\frac{p_m}{q_m}\right |\leq\frac{1}{n^32^nq_m^{n+1}},$
\end{enumerate}
as $q_m \to \infty.$
\end{cor}
\begin{proof}[\textbf{Proof}]  Let $\beta$ be a root of the polynomial $f(x)=a_0x^{n}+a_1x^{n-1}+\cdots+a_{n-1}x+a_n$ of degree $\deg f=n$. Since $x=p_m/q_m\approx 1$, the rational approximations are in the range $0<1-p_m/q_m\leq |x|\leq 1+p_m/q_m<2$. In addition, the coefficients $a_k$ coincides with the symmetric polynomials $s_k(\beta_1,\ldots, \beta_n)=a_k$ in terms of the roots $\beta=\beta_1,\ldots, \beta_n$ of the polynomial $f(x)$. Thus, $|a_k|\leq 2n$. In summary, the absolute value of the derivative satisfies
\begin{eqnarray}\label{eq2002.522}
|f^{\prime}(x)|&=&\left |na_0x^{n-1}+(n-1)a_1x^{n-2}+\cdots+a_{n-1}\right |\\
&<&n^32^n=A\nonumber,
\end{eqnarray}   
see {\color{red}\cite[Equation 7.4]{NI1956}} for more details.
\end{proof}

\begin{thm} \label{thm2002.530}\hypertarget{thm2002.530} Let $\beta$ be an algebraic integer of degree $n+1$, such that $\beta\approx 1$, and let $\{p_m/q_m:m\geq1\}$ be the sequence of convergents. Then,
\begin{equation}\label{eq2002.532}
\left |\beta-1\right |>\frac{1}{(n+1)^32^{n+1} p^{2(n+1)} },
\end{equation}
as $q_{m_p} \to \infty.$
\end{thm}

\begin{proof}[\textbf{Proof}]  Let $\beta=\alpha_p^{n+1} \overline{\alpha_p}^{-(n+1)}$. The inverse triangle inequality, and \hyperlink{cor2002.500}{Corollary} \ref{cor2002.500} lead to
\begin{eqnarray}\label{eq2002.534}
\left| \beta-1\right | &=&\left| \beta-1+ \frac{p_m}{q_m}-\frac{p_m}{q_m}\right |\\
&\geq&\left| \left |\beta- \frac{p_m}{q_m}\right |- \left| 1-\frac{p_m}{q_m}\right |\right |\nonumber \\
&\geq& \left |\beta- \frac{p_m}{q_m}\right |\nonumber \\
&\geq&\frac{1}{(n+1)^32^{n+1} q_m^{n+1}} \nonumber.
\end{eqnarray}
For large prime $p\geq 2$, let $p_{m_p}/q_{m_p}$ be a subsequence of convergents indexed by $p$ such that $q_{m_p}\leq p^2$. Then,
\begin{equation}\label{eq2002.536}
\left| \beta-1\right | >\frac{1}{(n+1)^32^{n+1} p^{2(n+1)} },
\end{equation}
as $q_{m_p} \to \infty.$
\end{proof}

%SSSSSSSSSSSSSSSSSSSSSSSSSSSSSSSSSSSSSSSSSSSSSSSSSSSSSSSSSSSSSSSSSSSSSSSSSSSSSSSSSSS
%SSSSSSSSSSSSSSSSSSSSSSSSSSSSSSSSSSSSSSSSSSSSSSSSSSSSSSSSSSSSSSSSSSSSSSSSSSSSSSSSSSS
%SSSSSSSSSSSSSSSSSSSSSSSSSSSSSSSSSSSSSSSSSSSSSSSSSSSSSSSSSSSSSSSSSSSSSSSSSSSSSSSSSSS
%SSSSSSSSSSSSSSSSSSSSSSSSSSSSSSSSSSSSSSSSSSSSSSSSSSSSSSSSSSSSSSSSSSSSSSSSSSSSSSSSSSS
%SSSSSSSSSSSSSSSSSSSSSSSSSSSSSSSSSSSSSSSSSSSSSSSSSSSSSSSSSSSSSSSSSSSSSSSSSSSSSSSSSSS
%SSSSSSSSSSSSSSSSSSSSSSSSSSSSSSSSSSSSSSSSSSSSSSSSSSSSSSSSSSSSSSSSSSSSSSSSSSSSSSSSSSS
%SSSSSSSSSSSSSSSSSSSSSSSSSSSSSSSSSSSSSSSSSSSSSSSSSSSSSSSSSSSSSSSSSSSSSSSSSSSSSSSSSSS
\section{Result For Linear Forms In Logarithms}\label{s2222}

\begin{thm}\label{thm2222.002} \hypertarget{thm2222.002}{\normalfont({\color{red}\cite[Theorem 3.1]{BA1975}})} Let $\alpha_1, \alpha_2, \ldots, \alpha_d$ be nonnegative algebraic numbers, and let $b_1,b_2, \ldots, b_d$ be rational integers. If 
\begin{equation} \label{eq2222.024}
\Lambda= b_1\log \alpha_1+ b_2\log \alpha_2^{b_2}+ \cdots +b_d\log \alpha_d\ne0,
\end{equation}
is not a null period, then
\begin{equation} \label{eq2222.024}
|\Lambda|>e^{-C},
\end{equation}
where
\begin{equation} \label{eq2222.026}
B\geq \max \left \{ | b_1|,| b_2|,\ldots ,| b_d| \right \} ,
\end{equation}
and
\begin{equation} \label{eq2222.028}
A_i\geq \max \left \{ dh(\alpha_i),| h(\alpha_i)|,0.16 \right \} 
\end{equation}
for $i=1,2, \ldots d$.
\end{thm}

Explicit versions of Baker theorem are widely used in Diophantine analysis. The height of a polynomial of degree $d=\deg f\geq 1$,  

\begin{equation} \label{eq2222.020}
f(x)=a_d\prod_{1\leq i\leq d}\left (x-\alpha_i \right ),
\end{equation} 
where $a_d>0$, is defined by 
\begin{equation} \label{eq2222.022}
h(\alpha)= \frac{1}{d} \left ( \log a_0+\sum_{1\leq i\leq d}\log \left (\max \{ | \alpha_i|,1 \}  \right )\right ).
\end{equation} 

\begin{thm} \label{thm2222.002} {\normalfont(Matveev)} Let $\alpha_1, \alpha_2, \ldots, \alpha_d$ be nonnegative algebraic numbers in a number field $\mathcal{K}$ of degree $k=[\mathcal{K}:\Q]$, and let $b_1,b_2, \ldots, b_d$ be rational integers. If 
\begin{equation} \label{eq2222.024}
\Lambda= \alpha_1^{b_1} \alpha_2^{b_2} \cdots \alpha_d^{b_d}-1
\end{equation}
is not zero, then
\begin{equation} \label{eq2222.024}
|\Lambda|>e^{-1.4\times30^{d+3}d^{4.5}k^2\left (1+\log d\right ) \left (1+\log B\right )A_1 A_2 \cdots A_d},
\end{equation}
where
\begin{equation} \label{eq2222.026}
B\geq \max \left \{ | b_1|,| b_2|,\ldots ,| b_d| \right \} ,
\end{equation}
and
\begin{equation} \label{eq2222.028}
A_i\geq \max \left \{ dh(\alpha_i),| h(\alpha_i)|,0.16 \right \} 
\end{equation}
for $i=1,2, \ldots d$.
\end{thm}

The version stated above appears in {\color{red}\cite[Theorem 2]{AT2019}}, and {\color{red}\cite[p.\ 4]{AF2012}}. Observe that
\begin{equation}\label{eq2222.030}
\left |\Lambda \right |= \left |\alpha_1^{b_1} \alpha_2^{b_2} \cdots \alpha_d^{b_d}-1\right |
\geq \left | \left | \alpha_1\right |^{b_1} \left |\alpha_2\right |^{b_2} \cdots \left |\alpha_d\right |^{b_d}-1\right 
|.
\end{equation}
Thus, the result is applicable to any algebraic integers $\alpha_1, \alpha_2, \ldots, \alpha_d$.\\

Under the condition $\alpha \ne 1$, an estimate of the lower bound of the expression $\alpha^n-1$ can be achieved via the basic Diophantine inequality
\begin{eqnarray}\label{eq2222.034}
\left| \alpha^n-1\right | &=&\left| \alpha^n-1+ \frac{p_m}{q_m}-\frac{p_m}{q_m}\right |\\
&\geq&\left| \left |\alpha^n- \frac{p_m}{q_m}\right |- \left| 1-\frac{p_m}{q_m}\right |\right |\nonumber \\
&\geq&\frac{1}{2q_{m+1}} \nonumber,
\end{eqnarray}
where $p_m/q_m$ is the sequence of convergents of the irrational number $\alpha^n$. But, \hyperlink{thm2222.002}{Theorem} \ref{thm2222.002} provides an explicit value for any fixed integer $n\geq1$.

\begin{lem}\label{lem2222.006} Let $\alpha_p$ be a root of the polynomial $X^2-\tau(p) X+p^{11}$. Then, 
\begin{equation}\label{eq2222.027}
|\Lambda_p(n)|=\left |  \alpha_p^{n+1} \overline{\alpha_p}^{-(n+1)}-1\right |>p^{-c_0\log n}
\end{equation}
where $c_0=6.8 \times 10^{10}$, and any integer $n\geq 1$.
\end{lem}

\begin{proof}[\textbf{Proof}]  Let $\alpha_p=p^{11/2}e^{i\theta_p}$, where $0\leq \theta_p\leq \pi$, be the root of the polynomial $f(X)=a_2X+a_1X+a_0=X^2-\tau(p) X+p^{11}$. Consider the algebraic integer
\begin{equation} \label{eq2222.024}
\Lambda_p(n)= \alpha_1^{b_1} \alpha_2^{b_2} \cdots \alpha_n^{b_d}-1= \alpha_p^{n+1} \overline{\alpha_p}^{-(n+1)}-1.
\end{equation}
Here, the parameters are:
\begin{enumerate}
\item $[\mathcal{K}:\Q]=k=2$
\item $\deg f=d=2$
\item $\displaystyle B\geq \max \left \{ | b_1|,| b_2|,\ldots ,| b_d| \right \}=n+1$
\item $\displaystyle h(\alpha)= \frac{1}{d} \left ( \log a_0+\sum_{1\leq i\leq d}\log \left (\max \{ | \alpha_i|,1 \}  \right )\right )\leq 2\log p^{11}$
\item $\displaystyle A_i\geq \max \left \{ dh(\alpha_i),| h(\alpha_i)|,0.16 \right \}=\log p^{11/2}$.
\end{enumerate}
The lower bound is
\begin{eqnarray}\label{eq2222.032}
|\Lambda_p(n)|&>&e^{-1.4\times30^{d+3}d^{4.5}k^2\left (1+\log d\right ) \left (1+\log B\right )A_1 A_2 \cdots A_d}\\
&\geq&e^{-c_0\left (\log n\right ) \left (\log p\right )}\nonumber\\
&\geq&p^{-c_0\log n}\nonumber,
\end{eqnarray} 
where $c_0=6.8 \times 10^{10}$.
\end{proof}

%SSSSSSSSSSSSSSSSSSSSSSSSSSSSSSSSSSSSSSSSSSSSSSSSSSSSSSSSSSSSSSSSSSSSSSSSSSSSSSSSSSS
%SSSSSSSSSSSSSSSSSSSSSSSSSSSSSSSSSSSSSSSSSSSSSSSSSSSSSSSSSSSSSSSSSSSSSSSSSSSSSSSSSSS
%SSSSSSSSSSSSSSSSSSSSSSSSSSSSSSSSSSSSSSSSSSSSSSSSSSSSSSSSSSSSSSSSSSSSSSSSSSSSSSSSSSS
%SSSSSSSSSSSSSSSSSSSSSSSSSSSSSSSSSSSSSSSSSSSSSSSSSSSSSSSSSSSSSSSSSSSSSSSSSSSSSSSSSSS
%SSSSSSSSSSSSSSSSSSSSSSSSSSSSSSSSSSSSSSSSSSSSSSSSSSSSSSSSSSSSSSSSSSSSSSSSSSSSSSSSSSS
%SSSSSSSSSSSSSSSSSSSSSSSSSSSSSSSSSSSSSSSSSSSSSSSSSSSSSSSSSSSSSSSSSSSSSSSSSSSSSSSSSSS
%SSSSSSSSSSSSSSSSSSSSSSSSSSSSSSSSSSSSSSSSSSSSSSSSSSSSSSSSSSSSSSSSSSSSSSSSSSSSSSSSSSS
\section{Structures Of Even And Odd Coefficients}\label{S7799}\hypertarget{S7799}
The form of the integers $m$ required for odd and prime values of the function $\tau(m)$ was initiated in \cite{LD1965}. A complete and simplified version was derived in \cite{LR2013}.

\begin{thm} \label{thm7799.002}\hypertarget{thm7799.002} {\normalfont({\color{red}\cite[Theorem 1]{LR2013}})} Let $m\geq 1$ be an integer such that $\tau(m)=p$ is an odd prime. Then, $m=q^{2n}$, where $q\geq 3$ and $2n+1\geq3$ are primes, and $p\nmid \tau(p)$.
\end{thm}

The nonvanishing $\tau(m)\ne0$ of the coefficients was initiated in \cite{LD1947}, currently it has been verified for approximately $m\leq 10^{23}$, see \cite{SJ2003}, {\color{red}\cite[p.\ 2]{DZ2013}}. 
%SSSSSSSSSSSSSSSSSSSSSSSSSSSSSSSSSSSSSSSSSSSSSSSSSSSSSSSSSSSSSSSSSSSSSSSSSSSSSSSSSSS
%SSSSSSSSSSSSSSSSSSSSSSSSSSSSSSSSSSSSSSSSSSSSSSSSSSSSSSSSSSSSSSSSSSSSSSSSSSSSSSSSSSS
%SSSSSSSSSSSSSSSSSSSSSSSSSSSSSSSSSSSSSSSSSSSSSSSSSSSSSSSSSSSSSSSSSSSSSSSSSSSSSSSSSSS
%SSSSSSSSSSSSSSSSSSSSSSSSSSSSSSSSSSSSSSSSSSSSSSSSSSSSSSSSSSSSSSSSSSSSSSSSSSSSSSSSSSS
%SSSSSSSSSSSSSSSSSSSSSSSSSSSSSSSSSSSSSSSSSSSSSSSSSSSSSSSSSSSSSSSSSSSSSSSSSSSSSSSSSSS
%SSSSSSSSSSSSSSSSSSSSSSSSSSSSSSSSSSSSSSSSSSSSSSSSSSSSSSSSSSSSSSSSSSSSSSSSSSSSSSSSSSS
%SSSSSSSSSSSSSSSSSSSSSSSSSSSSSSSSSSSSSSSSSSSSSSSSSSSSSSSSSSSSSSSSSSSSSSSSSSSSSSSSSSS
\section{Lower Bounds and Upper Bounds}\label{S9999ASC-C}\hypertarget{S9999ASC-C}
The form of the integers $n$ required for odd and prime values of the function $\tau(n)$ was initiated in \cite{LD1965}. A complete and simplified version was derived in \cite{LR2013}.

\begin{thm} \label{thm9999ASC.002}\hypertarget{thm9999ASC.002} {\normalfont({\color{red}\cite[Theorem 1]{LR2013}})} Let $n\geq 1$ be an integer such that $\tau(n)$ is an odd prime. Then, $n=q^{r-1}$, where $q\geq 3$ and $r\geq3$ are primes, and $p\nmid \tau(p)$.
\end{thm}
\begin{thm} \label{thm9999UBR.004}\hypertarget{thm9999UBR.004} {\normalfont (Deligne)} If $n$ is an integer and $d(n)$ is the divisor function, then
	\begin{equation} \label{eq9999UBR.036}
		\left |\tau\left (n\right )\right | \geq d(n)n^{11/2}.
	\end{equation} 
\end{thm}
For a prime, the Deligne theorem reduces to the upper bound 
\begin{equation} \label{eq9999ASC.020s}
	\left |\tau\left (p\right )\right | \leq 2p^{11/2}.
\end{equation}

\begin{conj} \label{conj9999ASC.010}\hypertarget{conj9999ASC.010} {\normalfont (Atkin-Serre)} If $\varepsilon>0$ is a small number and $n>0$ is an integer, then
	\begin{equation} \label{eq9999ASC.036}
		\left |\tau\left (n\right )\right | \geq n^{9/2-\varepsilon}
	\end{equation}
	where $c>0$ is a constant. 
\end{conj} 
The Atkin-Serre conjecture is true on a set of integers of density 1, see \cite{GW2021} and has extension to other modular forms, see \cite{RJ2007}. A weaker, almost trivial, but unconditional lower bound is given by the next result. 
\begin{thm} \label{thm9999E.004}\hypertarget{thm9999E.004} {\normalfont ({\color{red}\cite[Theorem 1]{MS2007}})} There exists a large effectively computable constant $c>0$ such that for all nonnegative integers $n$ for which $\tau(n)$ is odd, then
	\begin{equation} \label{eq9999.036}
		\left |\tau\left (n\right )\right | \geq (\log n)^c.
	\end{equation} 
\end{thm}

The large and effectively computable constant, using Matveev theorem or similar method, is known to be significantly larger than $c>10^6$. Jointly, these results satisfy the inequalities
\begin{equation} \label{eq9999ASC.024}
	(\log n)^c\leq p^{9/2-\varepsilon}\leq \left |\tau\left (p\right )\right | \leq 2p^{11/2}.
\end{equation}

%SSSSSSSSSSSSSSSSSSSSSSSSSSSSSSSSSSSSSSSSSSSSSSSSSSSSSSSSSSSSSSSSSSSSSSSSSSSSSSSSSSS
%SSSSSSSSSSSSSSSSSSSSSSSSSSSSSSSSSSSSSSSSSSSSSSSSSSSSSSSSSSSSSSSSSSSSSSSSSSSSSSSSSSS
%SSSSSSSSSSSSSSSSSSSSSSSSSSSSSSSSSSSSSSSSSSSSSSSSSSSSSSSSSSSSSSSSSSSSSSSSSSSSSSSSSSS
%SSSSSSSSSSSSSSSSSSSSSSSSSSSSSSSSSSSSSSSSSSSSSSSSSSSSSSSSSSSSSSSSSSSSSSSSSSSSSSSSSSS
%SSSSSSSSSSSSSSSSSSSSSSSSSSSSSSSSSSSSSSSSSSSSSSSSSSSSSSSSSSSSSSSSSSSSSSSSSSSSSSSSSSS
%SSSSSSSSSSSSSSSSSSSSSSSSSSSSSSSSSSSSSSSSSSSSSSSSSSSSSSSSSSSSSSSSSSSSSSSSSSSSSSSSSSS
%SSSSSSSSSSSSSSSSSSSSSSSSSSSSSSSSSSSSSSSSSSSSSSSSSSSSSSSSSSSSSSSSSSSSSSSSSSSSSSSSSSS
\section{Fourier Coefficients Estimate I}\label{s9977}
The lower bound considered here is based on the oldest result for the irrationality measure of algebraic real numbers. Nevertheless, it produces an effective result, of nearly the same strength as the Atkin-Serre conjecture, see \hyperlink{conj9999ASC.010}{Conjecture} \ref{conj9999ASC.010}. The concept of measures of irrationality of real numbers is discussed in {\color{red}\cite[Section 7.4]{NI1956}}, et alii.
 
\begin{thm} {\normalfont}\label{thm2277.106}\hypertarget{thm2277.106} Let $p\geq 2$ be a prime, and let $n\geq 1$ be an integer. Then,
\begin{equation}\label{eq9977.180}
\left |\tau(p^{n})\right |>A^{-1}p^{7n/2-2-\varepsilon},
\end{equation}
where $A=(n+1)^32^{n+1}$, and $\varepsilon>0$ is a small number.
\end{thm}

\begin{proof}[\textbf{Proof}]  Consider the algebraic number $\beta=\alpha_p^{n+1} \overline{\alpha_p}^{-(n+1)}$ of degree $n+1$. Replacing the estimate in \hyperlink{thm2002.530}{Theorem} \ref{thm2002.530} in equation \eqref{eq9977.082}, yields 
\begin{eqnarray} \label{eq9977.182}
\left |\tau(p^n)\right |
&\geq& p^{11n/2}\left | \alpha_p^{n+1} \overline{\alpha_p}^{-(n+1)}-1 \right |\\
&\geq& p^{11n/2}\cdot \frac{1}{Ap^{2(n+1)}}\nonumber,
\end{eqnarray}
where $A=(n+1)^32^{n+1}$, for all $p^{2n}\geq 10^{600}$.
\end{proof}

Inequality \eqref{eq9977.180} was checked against the known numerical data for both primes and probable primes, {\color{red}\cite[Table 3]{LR2013}}, every entry $\left |\tau\left (p^{2n}\right )\right |=\left |\tau\left (p^{q-1}\right )\right |$ is well above the estimated lower bound. 

\begin{exa}{\normalfont A few cases were computed to illustrate the inequality in previous result. 
\begin{enumerate}
\item The largest known prime coefficient $\left |\tau\left (p^{2n}\right )\right |=\left |\tau\left (p^{q-1}\right )\right |\approx 10^{d}$ has the parameters $p=157$, $q=2n+1=2207$, $d=26643$ digits, and $\varepsilon=1$. The inequality is
\begin{equation}
\left |\tau\left (p^{q-1}\right )\right |\approx 10^{26643}\geq \frac{p^{7n-2-\varepsilon}}{(2n+1)^32^{2n+1}}\approx 10^{16275}. 
\end{equation} 
\item The largest known probable prime coefficient $\left |\tau\left (p^{2n}\right )\right |=\left |\tau\left (p^{q-1}\right )\right |\approx 10^{d}$ has the parameters $p=41$, $q=2n+1=28289$, $d=250924$ digits, and $\varepsilon=1$. The inequality is
\begin{equation}
\left |\tau\left (p^{q-1}\right )\right |\approx 10^{250924}\geq \frac{p^{7n-2-\varepsilon}}{(2n+1)^32^{2n+1}}\approx 10^{151146}. 
\end{equation} 
\end{enumerate}
}
\end{exa}

%SSSSSSSSSSSSSSSSSSSSSSSSSSSSSSSSSSSSSSSSSSSSSSSSSSSSSSSSSSSSSSSSSSSSSSSSSSSSSSSSSSS
%SSSSSSSSSSSSSSSSSSSSSSSSSSSSSSSSSSSSSSSSSSSSSSSSSSSSSSSSSSSSSSSSSSSSSSSSSSSSSSSSSSS
%SSSSSSSSSSSSSSSSSSSSSSSSSSSSSSSSSSSSSSSSSSSSSSSSSSSSSSSSSSSSSSSSSSSSSSSSSSSSSSSSSSS
%SSSSSSSSSSSSSSSSSSSSSSSSSSSSSSSSSSSSSSSSSSSSSSSSSSSSSSSSSSSSSSSSSSSSSSSSSSSSSSSSSSS
%SSSSSSSSSSSSSSSSSSSSSSSSSSSSSSSSSSSSSSSSSSSSSSSSSSSSSSSSSSSSSSSSSSSSSSSSSSSSSSSSSSS
%SSSSSSSSSSSSSSSSSSSSSSSSSSSSSSSSSSSSSSSSSSSSSSSSSSSSSSSSSSSSSSSSSSSSSSSSSSSSSSSSSSS
%SSSSSSSSSSSSSSSSSSSSSSSSSSSSSSSSSSSSSSSSSSSSSSSSSSSSSSSSSSSSSSSSSSSSSSSSSSSSSSSSSSS
\section{Fourier Coefficients Estimate II}\label{s9977}
An explicit lower bound for the Tau function at the prime powers, based on linear forms in logarithm, is derived below. This lower bound is similar to {\color{red}\cite[Lemma 1]{MS1987}}, but a full proof is given here.
 
\begin{lem} {\normalfont}\label{lem2277.006} Let $p\geq 2$ be a prime, and let $n\geq 2$ be an integer. If $\tau(p)\ne0$, then
\begin{equation}
\left |\tau(p^n)\right |>p^{11n/2-c_0\log n},
\end{equation}
where $c_0=6.8 \times 10^{10}$.
\end{lem}

\begin{proof}[\textbf{Proof}]  Modify the Binet formula into the following form
\begin{eqnarray} \label{eq9977.080}
\tau(p^n)&=&\frac{\alpha_p^{n+1}-\overline{\alpha_p}^{n+1}}{\alpha_p-\overline{\alpha_p}}\\
&=&\frac{\alpha_p+\overline{\alpha_p}}{\alpha_p+\overline{\alpha_p}} \cdot\frac{\alpha_p^{n+1}-\overline{\alpha_p}^{n+1}}{\alpha_p-\overline{\alpha_p}}\nonumber\\
&=&\left (\alpha_p+\overline{\alpha_p}\right ) \cdot\frac{\alpha_p^{n+1}-\overline{\alpha_p}^{n+1}}{\alpha_p^2-\overline{\alpha_p}^2}\nonumber\\
&=&\left (\alpha_p+\overline{\alpha_p}\right ) \cdot\frac{\overline{\alpha_p}^{n+1}}{\alpha_p^2-\overline{\alpha_p}^2} \left (\alpha_p^{n+1} \overline{\alpha_p}^{-(n+1)}-1\right )\nonumber.
\end{eqnarray}

Replace the algebraic integer $\alpha_p=p^{11/2}e^{i\theta_p}$, where $0< \theta_p< \pi$. Taking absolute value and simplifying yield

\begin{eqnarray} \label{eq9977.082}
\left |\tau(p^n)\right |&=&\left |p^{11/2}\cos(\theta_p)\right |
\left |\frac{p^{11(n+1)/2}e^{i\theta_p}}{2p^{11}\sin(2\theta_p)}\right |
\left | \alpha_p^{n+1} \overline{\alpha_p}^{-(n+1)}-1 \right | \\
&\geq& p^{11n/2}\left | \alpha_p^{n+1} \overline{\alpha_p}^{-(n+1)}-1 \right |\nonumber,
\end{eqnarray}
since $\tau(p)\ne0$ implies $\theta\ne\pi/2$.
Replacing the estimate in Lemma \ref{lem2222.006} for $n\geq 1$, yields
\begin{eqnarray} \label{eq9977.084}
p^{11n/2}\left | \alpha_p^{n+1} \overline{\alpha_p}^{-(n+1)}-1 \right |
&\geq & p^{11n/2}\cdot \left( e^{-c_0 \left (\log n\right ) \left (\log p\right )}\right ) 
\\
&\geq & p^{11n/2-c_0\log n} \nonumber,
\end{eqnarray}
where $c_0=6.8 \times 10^{10}$.
\end{proof}

%SSSSSSSSSSSSSSSSSSSSSSSSSSSSSSSSSSSSSSSSSSSSSSSSSSSSSSSSSSSSSSSSSSSSSSSSSSSSSSSSSSS
%SSSSSSSSSSSSSSSSSSSSSSSSSSSSSSSSSSSSSSSSSSSSSSSSSSSSSSSSSSSSSSSSSSSSSSSSSSSSSSSSSSS
%SSSSSSSSSSSSSSSSSSSSSSSSSSSSSSSSSSSSSSSSSSSSSSSSSSSSSSSSSSSSSSSSSSSSSSSSSSSSSSSSSSS
%SSSSSSSSSSSSSSSSSSSSSSSSSSSSSSSSSSSSSSSSSSSSSSSSSSSSSSSSSSSSSSSSSSSSSSSSSSSSSSSSSSS
%SSSSSSSSSSSSSSSSSSSSSSSSSSSSSSSSSSSSSSSSSSSSSSSSSSSSSSSSSSSSSSSSSSSSSSSSSSSSSSSSSSS
%SSSSSSSSSSSSSSSSSSSSSSSSSSSSSSSSSSSSSSSSSSSSSSSSSSSSSSSSSSSSSSSSSSSSSSSSSSSSSSSSSSS
%SSSSSSSSSSSSSSSSSSSSSSSSSSSSSSSSSSSSSSSSSSSSSSSSSSSSSSSSSSSSSSSSSSSSSSSSSSSSSSSSSSS
\section{Explicit Lower Bound}
A different unconditional and completely explicit lower bound is proved in this Section.

\begin{thm} \label{thm9999.050} \hypertarget{thm9999.050}Let $n\geq 1$ be an integer, and let $p\geq 2$ be a prime. If $p^{2n}\geq 10^{600}$, then,
\begin{equation} \label{eq9999.050}
\left |\tau\left (p^{2n}\right )\right | \geq (\log p^{2n})^{10}.
\end{equation} 
\end{thm}
\begin{proof}[\textbf{Proof}]  On the contrary, suppose that
\begin{equation}\label{eq9999.058}
(\log p^{2n})^{10}> \left |\tau\left (p^{2n}\right )\right |,
\end{equation}
for $p^{2n}\geq 10^{600}$. By \hyperlink{thm2277.106}{Theorem} \ref{thm2277.106},
\begin{eqnarray} \label{eq9999.060}
(\log p^{2n})^{10}
&>& \left |\tau\left (p^{2n}\right )\right |\\
&\geq& A^{-1}p^{7n-2-\varepsilon}\nonumber,
\end{eqnarray}
where $A=(2n+1)^32^{2n+1}$ is a constant. Equivalently,
\begin{equation}\label{eq9999.058}
(2n+1)^32^{2n+1}(\log p^{2n})^{10}
>p^{7n-2-\varepsilon}
\end{equation}
Taking logarithm across the board yields
\begin{equation} \label{eq9999.060}
3\log (2n+1)+(2n+1)\log 2+10(\log n+\log \log p+\log 2)>(7n-2-\varepsilon)\log p.
\end{equation}
Division by $n\log p$ yields
\begin{eqnarray} \label{eq9999.060}
\frac{3\log (2n+1)+(2n+1)\log 2+10(\log n+\log \log p+\log 2)}{n\log p}&>&7-\frac{2+\varepsilon}{n}\nonumber\\
&>& 6.
\end{eqnarray}
Numerical calculations shows that this false for $p^{2n}\geq 10^{600}$, where $2n\geq2$, $p\geq2$, and a small number $\varepsilon\leq 1$.
\end{proof}

Inequality \eqref{eq9999.050} was checked against the known numerical data for both primes and probable primes, {\color{red}\cite[Table 3]{LR2013}}, every entry $\left |\tau\left (p^{2n}\right )\right |=\left |\tau\left (p^{q-1}\right )\right |$ is well above the estimated lower bound.

%SSSSSSSSSSSSSSSSSSSSSSSSSSSSSSSSSSSSSSSSSSSSSSSSSSSSSSSSSSSSSSSSSSSSSSSSSSSSSSSSSSS
%SSSSSSSSSSSSSSSSSSSSSSSSSSSSSSSSSSSSSSSSSSSSSSSSSSSSSSSSSSSSSSSSSSSSSSSSSSSSSSSSSSS
%SSSSSSSSSSSSSSSSSSSSSSSSSSSSSSSSSSSSSSSSSSSSSSSSSSSSSSSSSSSSSSSSSSSSSSSSSSSSSSSSSSS
%SSSSSSSSSSSSSSSSSSSSSSSSSSSSSSSSSSSSSSSSSSSSSSSSSSSSSSSSSSSSSSSSSSSSSSSSSSSSSSSSSSS
%SSSSSSSSSSSSSSSSSSSSSSSSSSSSSSSSSSSSSSSSSSSSSSSSSSSSSSSSSSSSSSSSSSSSSSSSSSSSSSSSSSS
%SSSSSSSSSSSSSSSSSSSSSSSSSSSSSSSSSSSSSSSSSSSSSSSSSSSSSSSSSSSSSSSSSSSSSSSSSSSSSSSSSSS
%SSSSSSSSSSSSSSSSSSSSSSSSSSSSSSSSSSSSSSSSSSSSSSSSSSSSSSSSSSSSSSSSSSSSSSSSSSSSSSSSSSS
\section{Small Prime Values Results } \label{S9099SPV-R} \hypertarget{S9099SPV-R}
The proof of \hyperlink{thm9999.001}{Theorem} \ref{thm9999.001} is assembled in this section.
\begin{proof}[\textbf{Proof}] Let $q\leq 8.0 \times 10^{25}< \tau(251^2)$ be a fixed prime. The verification that 
\begin{equation} \label{eq9999.025}
\tau\left (p^{2n}\right )=\pm q
\end{equation}
has no solutions $p^{2n} \in \mathbb{N}$ splits into two cases.\\

\textit{Case {\normalfont I}. Small Solutions $p^{2n}\leq 10^{600}$.} By \hyperlink{thm7799.002}{Theorem} \ref{thm7799.002}, the prime values $\tau(p^{2n})=\pm q$ require integers of the form $p^{2n}$, where $p\geq 3$ and $2n+1\geq 3$ are primes. By the numerical works in \cite{LD1965} and {\color{red}\cite[Table 2]{LR2013}} for the range $p\leq10^6$ and $2n+1\leq 100$, it follows that equation \eqref{eq9999.025} has no small integer solutions $p^{2n}\leq 10^{600}$.\\

\textit{Case {\normalfont II}. Large Solutions $p^{2n}\geq 10^{600}$.} Assume that equation \eqref{eq9999.025} has a large solution. By Theorem \ref{thm9999.050},  
\begin{eqnarray} \label{eq9999.033}
\left |\tau\left (p^{2n}\right )\right | 
&\geq& (\log p^{2n})^{10}\\
&\geq& \left (\log 10^{600}\right )^{10}\nonumber\\
&\geq& 2.5231\times 10^{31}\nonumber.
\end{eqnarray}
In particular, for the small bounded subset of primes $q$, it reduces to
\begin{eqnarray} \label{eq9999.037}
8.0\times 10^{25}&\geq&q\\
&=&\left |\tau\left (p^{2n}\right )\right | \nonumber\\
&\geq& 2.5231\times 10^{31}\nonumber.
\end{eqnarray}

Clearly, this is false. Hence, equation \eqref{eq9999.025} has no large integer solutions $p^{2n}\geq 10^{600}$.
\end{proof}

%SSSSSSSSSSSSSSSSSSSSSSSSSSSSSSSSSSSSSSSSSSSSSSSSSSSSSSSSSSSSSSSSSSSSSSSSSSSSSSSSSSS
%SSSSSSSSSSSSSSSSSSSSSSSSSSSSSSSSSSSSSSSSSSSSSSSSSSSSSSSSSSSSSSSSSSSSSSSSSSSSSSSSSSS
%SSSSSSSSSSSSSSSSSSSSSSSSSSSSSSSSSSSSSSSSSSSSSSSSSSSSSSSSSSSSSSSSSSSSSSSSSSSSSSSSSSS
%SSSSSSSSSSSSSSSSSSSSSSSSSSSSSSSSSSSSSSSSSSSSSSSSSSSSSSSSSSSSSSSSSSSSSSSSSSSSSSSSSSS
%SSSSSSSSSSSSSSSSSSSSSSSSSSSSSSSSSSSSSSSSSSSSSSSSSSSSSSSSSSSSSSSSSSSSSSSSSSSSSSSSSSS
%SSSSSSSSSSSSSSSSSSSSSSSSSSSSSSSSSSSSSSSSSSSSSSSSSSSSSSSSSSSSSSSSSSSSSSSSSSSSSSSSSSS
%SSSSSSSSSSSSSSSSSSSSSSSSSSSSSSSSSSSSSSSSSSSSSSSSSSSSSSSSSSSSSSSSSSSSSSSSSSSSSSSSSSS
\section{Smallest Prime Values} \label{S9999SPV-T}\hypertarget{S9999SPV-T}
The absolute smallest prime value is 
\begin{equation} \label{eq9999.105b}
	\tau\left (251^2\right )=-80561663527802406257321747,
\end{equation}
the primality test analysis appears in \cite{LD1965}. Recent computer studies have unearthed more primes. Some information on for the first few prime values are tabulated below. A larger table appears in {\color{red}\cite[Table 3]{LR2013}}.\\

%TTTTTTTTTTTTTTTTTTTTTTTTTTTTTTTTTTTTTTTTTTTTTTTTTTTTTTTTTTTTTTTTTTTTTTTTTTTTT
\vskip .1 in 
\begin{table}[H]
	\setlength{\tabcolsep}{0.30cm}
	\renewcommand{\arraystretch}{1.250092}
	\setlength{\arrayrulewidth}{0.95pt}
	\centering
	\begin{tabular}{c|c|c|c}
		$\tau(p^{2n})$&Number of Decimal Digits&$\tau(p^{2n})$&Number of Decimal Digits\\
		\hline
		$\tau(251^2)$ & 26&$\tau(677^2)$ & 32\\
		\hline
		$\tau(971^2)$ & 33&$\tau(983^2)$ & 33\\
		\hline
		$\tau(47^4)$ & 37&$\tau(197^4)$ & 50\\ 
	\end{tabular}
	\vskip .1 in		
	\caption{Smallest Prime Values of $\tau(p^{2n})$} \label{table-901}
\end{table}
%TTTTTTTTTTTTTTTTTTTTTTTTTTTTTTTTTTTTTTT

%NNNNNNNNNNNNNNNNNNNNNNNNNNNNNNNNNNNNNNNNNNNNNNNNNNNNNNNNNNNNNNNNNNN
%NNNNNNNNNNNNNNNNNNNNNNNNNNNNNNNNNNNNNNNNNNNNNNNNNNNNNNNNNNNNNNNNNNN
%NNNNNNNNNNNNNNNNNNNNNNNNNNNNNNNNNNNNNNNNNNNNNNNNNNNNNNNNNNNNNNNNNNN
%NNNNNNNNNNNNNNNNNNNNNNNNNNNNNNNNNNNNNNNNNNNNNNNNNNNNNNNNNNNNNNNNNNN
%NNNNNNNNNNNNNNNNNNNNNNNNNNNNNNNNNNNNNNNNNNNNNNNNNNNNNNNNNNNNNNNNNNN
\newpage
\textbf{\Large Appendix}

%SSSSSSSSSSSSSSSSSSSSSSSSSSSSSSSSSSSSSSSSSSSSSSSSSSSSSSSSSSSSSSSSSSSSSSSSSSSSSSSSSSS
%SSSSSSSSSSSSSSSSSSSSSSSSSSSSSSSSSSSSSSSSSSSSSSSSSSSSSSSSSSSSSSSSSSSSSSSSSSSSSSSSSSS
%SSSSSSSSSSSSSSSSSSSSSSSSSSSSSSSSSSSSSSSSSSSSSSSSSSSSSSSSSSSSSSSSSSSSSSSSSSSSSSSSSSS
%SSSSSSSSSSSSSSSSSSSSSSSSSSSSSSSSSSSSSSSSSSSSSSSSSSSSSSSSSSSSSSSSSSSSSSSSSSSSSSSSSSS
%SSSSSSSSSSSSSSSSSSSSSSSSSSSSSSSSSSSSSSSSSSSSSSSSSSSSSSSSSSSSSSSSSSSSSSSSSSSSSSSSSSS
%SSSSSSSSSSSSSSSSSSSSSSSSSSSSSSSSSSSSSSSSSSSSSSSSSSSSSSSSSSSSSSSSSSSSSSSSSSSSSSSSSSS
%SSSSSSSSSSSSSSSSSSSSSSSSSSSSSSSSSSSSSSSSSSSSSSSSSSSSSSSSSSSSSSSSSSSSSSSSSSSSSSSSSSS
\section{Fourier Coefficients Formulas}\label{S9922}\hypertarget{S9922}
Some of the theory of modular forms 
\begin{equation} \label{eq9999.003}
f(s)=\sum_{n \geq1} \lambda(n)q^n=\prod_{p \mid N} \left (\frac{1}{1-p^{(k-1)s}}\right )\prod_{p \nmid N} \left (\frac{1}{1-\lambda(p)-p^{(k-1)s}}\right ),
\end{equation}
where $s\in \mathbb{H}$ is a complex number in the upper half plane, $q=e^{i2\pi }$, are introduced in \cite{KN1984}, et alii. 
\begin{lem} {\normalfont(\cite{ML1917})}\label{lem2222.002}\hypertarget{lem2222.002} Let $m,n\geq 1$ be a pair of integers, and let $p\geq 2$ be a prime. If $\lambda(n)$ is the $n$th Fourier coefficient of a modular form $f(z)=\sum_{n\geq 1}\lambda(n)q^n$ of weight $k\geq 4$, then
\begin{enumerate}[font=\normalfont, label=(\roman*)]
\item $\displaystyle \lambda(mn)=\lambda(m)\lambda(n)$ if $\gcd(m,n)=1$.
\item $\displaystyle \lambda(p^n)=\lambda(p)\lambda(p^{n-1})-p^{(k-1)/2}\lambda(p^{n-2})$ if $n\geq 2$.
\end{enumerate}
\end{lem}
Finer details appears in {\color{red}\cite[Proposition 32]{KN1984}} and similar references. These are compactly expressed by Mordell formula
\begin{equation}\label{eq2222.040}
\lambda(m)\lambda(n)=\sum_{d\mid \gcd(m,n)}d^{k-1} \lambda(mn/d^2).
\end{equation}

\begin{lem} \label{lem2222.044}\hypertarget{lem2222.044} For any prime $p\geq 2$, and an integer $n\geq 1$, $\lambda(p^n)\in \Z[p^{(k-1)/2},\lambda(p)]$ has an expansion as a polynomial
\begin{equation}\label{eq2222.040}
\lambda(p^n)=\sum_{0\leq k\leq n/2}(-1)^k \binom{n-k}{n-2k}p^{11k}\lambda(p)^{n-2k}
\end{equation}
of one variable $\lambda(p)$ of degree $n$.
\end{lem}

\begin{dfn}\label{dfn2222.000}\hypertarget{dfn2222.000} A weight $k\geq2$ coefficient $\lambda(p^n)$ has a representation as
\begin{equation}
\lambda\left (p^{2n}\right )=2p^{(k-1)n}\cos \theta, 
\end{equation}
where $0\leq \theta<\pi$.
\end{dfn}

%SSSSSSSSSSSSSSSSSSSSSSSSSSSSSSSSSSSSSSSSSSSSSSSSSSSSSSSSSSSSSSSSSSSSSSSSSSSSSSSSSSS
%SSSSSSSSSSSSSSSSSSSSSSSSSSSSSSSSSSSSSSSSSSSSSSSSSSSSSSSSSSSSSSSSSSSSSSSSSSSSSSSSSSS
%SSSSSSSSSSSSSSSSSSSSSSSSSSSSSSSSSSSSSSSSSSSSSSSSSSSSSSSSSSSSSSSSSSSSSSSSSSSSSSSSSSS
%SSSSSSSSSSSSSSSSSSSSSSSSSSSSSSSSSSSSSSSSSSSSSSSSSSSSSSSSSSSSSSSSSSSSSSSSSSSSSSSSSSS
%SSSSSSSSSSSSSSSSSSSSSSSSSSSSSSSSSSSSSSSSSSSSSSSSSSSSSSSSSSSSSSSSSSSSSSSSSSSSSSSSSSS
%SSSSSSSSSSSSSSSSSSSSSSSSSSSSSSSSSSSSSSSSSSSSSSSSSSSSSSSSSSSSSSSSSSSSSSSSSSSSSSSSSSS
%SSSSSSSSSSSSSSSSSSSSSSSSSSSSSSSSSSSSSSSSSSSSSSSSSSSSSSSSSSSSSSSSSSSSSSSSSSSSSSSSSSS
\section{Examples of New Form Product Expansion of Weight $ k $}\label{S7722NF-F}\hypertarget{S7722NF-F}
A few examples of newform of various weights and conductors are provided in this section.
%SSSSSSSSSSSSSSSSSSSSSSSSSSSSSSSSSSSSSSSSSSSSSSSSSSSSSSSSSSSSSSSSSSSSSSSSSSSSSSSSSSS
%SSSSSSSSSSSSSSSSSSSSSSSSSSSSSSSSSSSSSSSSSSSSSSSSSSSSSSSSSSSSSSSSSSSSSSSSSSSSSSSSSSS
%SSSSSSSSSSSSSSSSSSSSSSSSSSSSSSSSSSSSSSSSSSSSSSSSSSSSSSSSSSSSSSSSSSSSSSSSSSSSSSSSSSS
%SSSSSSSSSSSSSSSSSSSSSSSSSSSSSSSSSSSSSSSSSSSSSSSSSSSSSSSSSSSSSSSSSSSSSSSSSSSSSSSSSSS
%SSSSSSSSSSSSSSSSSSSSSSSSSSSSSSSSSSSSSSSSSSSSSSSSSSSSSSSSSSSSSSSSSSSSSSSSSSSSSSSSSSS
%SSSSSSSSSSSSSSSSSSSSSSSSSSSSSSSSSSSSSSSSSSSSSSSSSSSSSSSSSSSSSSSSSSSSSSSSSSSSSSSSSSS
%SSSSSSSSSSSSSSSSSSSSSSSSSSSSSSSSSSSSSSSSSSSSSSSSSSSSSSSSSSSSSSSSSSSSSSSSSSSSSSSSSSS
\subsection{New Form Product Expansion of Weight $ k=12 $}\label{s7722-1}

Weight $ k=12 $ and conductor $ N=1 $.
\begin{eqnarray} \label{eq9922.003}
	f(s)&=&\sum_{n \geq1} \lambda(n)q^{n}\\
	&=&q\prod_{n \geq 1} \left (1-q^{n}\right )^{24}\nonumber\\
	&=&q-24q^2+252q^3+\cdots\nonumber,
\end{eqnarray}
where $ q=e^{i\pi s} $, and $ s\in \mathcal{H}=\{s\in \C:\Re e(s)>0\} $.
%SSSSSSSSSSSSSSSSSSSSSSSSSSSSSSSSSSSSSSSSSSSSSSSSSSSSSSSSSSSSSSSSSSSSSSSSSSSSSSSSSSS
%SSSSSSSSSSSSSSSSSSSSSSSSSSSSSSSSSSSSSSSSSSSSSSSSSSSSSSSSSSSSSSSSSSSSSSSSSSSSSSSSSSS
%SSSSSSSSSSSSSSSSSSSSSSSSSSSSSSSSSSSSSSSSSSSSSSSSSSSSSSSSSSSSSSSSSSSSSSSSSSSSSSSSSSS
%SSSSSSSSSSSSSSSSSSSSSSSSSSSSSSSSSSSSSSSSSSSSSSSSSSSSSSSSSSSSSSSSSSSSSSSSSSSSSSSSSSS
%SSSSSSSSSSSSSSSSSSSSSSSSSSSSSSSSSSSSSSSSSSSSSSSSSSSSSSSSSSSSSSSSSSSSSSSSSSSSSSSSSSS
%SSSSSSSSSSSSSSSSSSSSSSSSSSSSSSSSSSSSSSSSSSSSSSSSSSSSSSSSSSSSSSSSSSSSSSSSSSSSSSSSSSS
%SSSSSSSSSSSSSSSSSSSSSSSSSSSSSSSSSSSSSSSSSSSSSSSSSSSSSSSSSSSSSSSSSSSSSSSSSSSSSSSSSSS
\subsection{New Form Product Expansion of Weight $ k=2 $}\label{s7722-11}

Weight $ k=2 $ and Conductor $ N=11 $.\\

The corresponding elliptic curve has the algebraic equation $ E:y^2-y=x^3-x^2 $. 
\begin{eqnarray} \label{eq9922.1103}
	f(s)&=&\sum_{n \geq1} \lambda(n)q^{n}\\
	&=&q\prod_{n \geq 1} \left (1-q^{n}\right )^{2}   \left (1-q^{11n}\right )^{2} \nonumber \\
	&=&q-2q^2-q^3+\cdots \nonumber,
\end{eqnarray}
where $ q=e^{i\pi s} $, and $ s\in \mathcal{H}=\{s\in \C:\Re e(s)>0\} $.
%SSSSSSSSSSSSSSSSSSSSSSSSSSSSSSSSSSSSSSSSSSSSSSSSSSSSSSSSSSSSSSSSSSSSSSSSSSSSSSSSSSS
%SSSSSSSSSSSSSSSSSSSSSSSSSSSSSSSSSSSSSSSSSSSSSSSSSSSSSSSSSSSSSSSSSSSSSSSSSSSSSSSSSSS
%SSSSSSSSSSSSSSSSSSSSSSSSSSSSSSSSSSSSSSSSSSSSSSSSSSSSSSSSSSSSSSSSSSSSSSSSSSSSSSSSSSS
%SSSSSSSSSSSSSSSSSSSSSSSSSSSSSSSSSSSSSSSSSSSSSSSSSSSSSSSSSSSSSSSSSSSSSSSSSSSSSSSSSSS
%SSSSSSSSSSSSSSSSSSSSSSSSSSSSSSSSSSSSSSSSSSSSSSSSSSSSSSSSSSSSSSSSSSSSSSSSSSSSSSSSSSS
%SSSSSSSSSSSSSSSSSSSSSSSSSSSSSSSSSSSSSSSSSSSSSSSSSSSSSSSSSSSSSSSSSSSSSSSSSSSSSSSSSSS
%SSSSSSSSSSSSSSSSSSSSSSSSSSSSSSSSSSSSSSSSSSSSSSSSSSSSSSSSSSSSSSSSSSSSSSSSSSSSSSSSSSS

\subsection{New Form Product Expansion of Weight $ k=2 $}\label{s7722-32}
Weight $ k=2 $ and Conductor $ N=32 $.\\

The corresponding elliptic curve has the algebraic equation $ E:y^2=x^3-x $. 
\begin{eqnarray} \label{eq9922.3203}
	f(s)&=&\sum_{n \geq1} \lambda(n)q^{n}\\
	&=&q\prod_{n \geq 1} \left (1-q^{4n}\right )^{2}   \left (1-q^{8n}\right )^{2}\nonumber\\ &=&q-2q^5-3q^9+6q^{13}+2q^{17}+\cdots \nonumber,
\end{eqnarray}
where $ q=e^{i\pi s} $, and $ s\in \mathcal{H}=\{s\in \C:\Re e(s)>0\} $.

%SSSSSSSSSSSSSSSSSSSSSSSSSSSSSSSSSSSSSSSSSSSSSSSSSSSSSSSSSSSSSSSSSSSSSSSSSSSSSSSSSSS
%SSSSSSSSSSSSSSSSSSSSSSSSSSSSSSSSSSSSSSSSSSSSSSSSSSSSSSSSSSSSSSSSSSSSSSSSSSSSSSSSSSS
%SSSSSSSSSSSSSSSSSSSSSSSSSSSSSSSSSSSSSSSSSSSSSSSSSSSSSSSSSSSSSSSSSSSSSSSSSSSSSSSSSSS
%SSSSSSSSSSSSSSSSSSSSSSSSSSSSSSSSSSSSSSSSSSSSSSSSSSSSSSSSSSSSSSSSSSSSSSSSSSSSSSSSSSS
%SSSSSSSSSSSSSSSSSSSSSSSSSSSSSSSSSSSSSSSSSSSSSSSSSSSSSSSSSSSSSSSSSSSSSSSSSSSSSSSSSSS
%SSSSSSSSSSSSSSSSSSSSSSSSSSSSSSSSSSSSSSSSSSSSSSSSSSSSSSSSSSSSSSSSSSSSSSSSSSSSSSSSSSS
%SSSSSSSSSSSSSSSSSSSSSSSSSSSSSSSSSSSSSSSSSSSSSSSSSSSSSSSSSSSSSSSSSSSSSSSSSSSSSSSSSSS
\section{Binet Formula}\label{S9922BF}\hypertarget{S9922BF}
A representation of the Fourier coefficients as a linear recurrent sequence is provided here. Various other interesting results on this topic appears in \cite{LR2013}.
\begin{dfn}\label{dfn9922BF.050}\hypertarget{dfn9922BF.050} The weight $k\geq2$ characteristic polynomial is defined by 
	\begin{equation}\label{eq9922BF.050j}
		f_k(X)=X^2-\lambda(p) X+p^{(k-1)/2}.
	\end{equation}
	Moreover, there is a unique root $\alpha_p =p^{(k-1)/2}e^{i\theta}$, where $0\leq \theta<\pi$.
\end{dfn}
\begin{lem} \label{lem9922BF.180}\hypertarget{lem9922BF.180} Let $f(z)=\sum_{n \geq1} \lambda(n)q^n$  be a modular form of level $N=1$ and weight $k\geq 4$. The $n$th coefficient $\lambda(p^n)$ at the prime power $p^n$ has the followings representations.
	\begin{enumerate}[font=\normalfont, label=(\roman*)]
		\item $\displaystyle \lambda(p^n)=2p^{(k-1)n/2}\cos n\theta$,\\
		
		\item $\displaystyle \lambda(p^n)=p^{(k-1)n/2}\frac{\sin (n+1) \theta}{\sin  \theta}$,
	\end{enumerate}
	where $0\leq \theta\leq \pi$.
\end{lem}

\begin{proof}[\textbf{Proof}]   (ii) Substitute $\alpha_p =p^{(k-1)/2}e^{i\theta}$ into the Binet formula for recurrent sequence defined by the polynomial \eqref{eq9922BF.050j}, and simplify it:
	\begin{eqnarray} \label{eq9922BF.180k}
		\lambda(p^n)&=&\frac{\alpha_p^{n+1}-\overline{\alpha_p}^{n+1}}{\alpha_p-\overline{\alpha_p}}\\
		&=&\frac{i2p^{(k-1)(n+1)/2}\sin (n+1) \theta}{i2p^{(k-1)/2}\sin  \theta}\nonumber \nonumber,
	\end{eqnarray}
	as required.
\end{proof}

%SSSSSSSSSSSSSSSSSSSSSSSSSSSSSSSSSSSSSSSSSSSSSSSSSSSSSSSSSSSSSSSSSSSSSSSSSSSSSSSSSSS
%SSSSSSSSSSSSSSSSSSSSSSSSSSSSSSSSSSSSSSSSSSSSSSSSSSSSSSSSSSSSSSSSSSSSSSSSSSSSSSSSSSS
%SSSSSSSSSSSSSSSSSSSSSSSSSSSSSSSSSSSSSSSSSSSSSSSSSSSSSSSSSSSSSSSSSSSSSSSSSSSSSSSSSSS
%SSSSSSSSSSSSSSSSSSSSSSSSSSSSSSSSSSSSSSSSSSSSSSSSSSSSSSSSSSSSSSSSSSSSSSSSSSSSSSSSSSS
%SSSSSSSSSSSSSSSSSSSSSSSSSSSSSSSSSSSSSSSSSSSSSSSSSSSSSSSSSSSSSSSSSSSSSSSSSSSSSSSSSSS
%SSSSSSSSSSSSSSSSSSSSSSSSSSSSSSSSSSSSSSSSSSSSSSSSSSSSSSSSSSSSSSSSSSSSSSSSSSSSSSSSSSS
%SSSSSSSSSSSSSSSSSSSSSSSSSSSSSSSSSSSSSSSSSSSSSSSSSSSSSSSSSSSSSSSSSSSSSSSSSSSSSSSSSSS
\section{Multiplicative Property}\label{S9922TCM12}\hypertarget{S9922TCM12}
The multiplicative property of the $n$th coefficients $\tau(n)$ was proved by Mordell. More recently, a partial proof of this result is discussed in \cite{WK2015}.

\begin{lem} {\normalfont(\cite{ML1917}) }\label{lem9922TCM.012}\hypertarget{lem9922TCM.012} Let $m,n\geq 1$ be a pair of integers, and let $p\geq 2$ be a prime. Then
	\begin{enumerate}[font=\normalfont, label=(\roman*)]
	\item $\displaystyle \tau(mn)=\tau(m)\tau(n)$ if $\gcd(m,n)=1$.
	\item $\displaystyle \tau(p^n)=\tau(p)\tau(p^{n-1})-p^{k-1}\tau(p^{n-2})$ if $n\geq 2$.
\end{enumerate}
\end{lem}

\begin{exa} \label{exa9922TCM12.044d}\hypertarget{exa9922TCM12.044d}{\normalfont The coefficient $\tau(p^n)$ is a polynomial in $\tau(n)$. The first few polynomials are these:
		\begin{enumerate}
			\item $\displaystyle \tau(p)$
			\item $\displaystyle \tau(p^2)=\tau(p)^2-p^{11}$
			\item $\displaystyle \tau(p^3)=\tau(p)\tau(p^2)-p^{11}\tau(p)=\tau(p)\left( \tau(p)^2-p^{11}\right )-p^{11}\tau(p)=\tau(p)^3-2p^{11}\tau(p)$.
			\item $\displaystyle \tau(p^4)=\tau(p)\tau(p^3)-p^{11}\tau(p^2)=\tau(p)^4-3p^{11}\tau(p)^2+p^{22}$.
			\item $\displaystyle \tau(p^5)=\tau(p)\tau(p^4)-p^{11}\tau(p^3)=\tau(p)^5-4p^{11}\tau(p)^3+3p^{22}\tau(p)$.
\item $\displaystyle			\tau(p^6)=\tau(p)\tau(p^5)-p^{11}\tau(p^4)=\tau(p)^6-5p^{11}\tau(p)^4 +6p^{11}\tau(p)^2-p^{33}$.
		\end{enumerate}	
These are recursively computed, for example, the last one is:
\begin{align}
	\tau(p^n)&=\tau(p)\tau(p^{n-1})-p^{k-1}\tau(p^{n-2})\\[.3cm]
	\tau(p^6)&=\tau(p)\tau(p^5)-p^{11}\tau(p^4)\nonumber\\[.3cm]
	&=\tau(p)\left( \tau(p)^5-4p^{11}\tau(p)^3+3p^{22}\tau(p)\right) -p^{11}\left(\tau(p)^4-3p^{11}\tau(p)^2+p^{22} \right)\nonumber\\[.3cm]
	&= \tau(p)^6-5p^{11}\tau(p)^4 +6p^{11}\tau(p)^2-p^{33}\nonumber. 
\end{align}
	}
\end{exa}

\begin{lem} \label{lem9922TCM.012-D}\hypertarget{lem9922TCM.012-D} Suppose that $\tau(p)=0$ for some prime $p>p_0$, then
	\begin{enumerate}[font=\normalfont, label=(\roman*)]
		\item $\displaystyle \tau(p^{2n+1})=0$,\tabto{7cm} for all integers $n\geq0$.
	\item $\displaystyle \tau(p^{2n})=\pm p^{11n}$,\tabto{7cm} for all integers $n\geq1$.
	\end{enumerate}
\end{lem}

\begin{proof}[\textbf{Proof}]   (i) By \hyperlink{lem9922TCM.012}{Lemma} \ref{lem9922TCM.012}, there is a polynomial $F(t)\in\Z[t] $ such that
\begin{equation}\label{eq9922TCM.012d2}
	F(\tau(p^{2n+1}))=\tau(p)F_0(\tau(p))
\end{equation}for any integer $n\geq0$, so $F(t)=0$ at $t=\tau(p)=0$. (ii) Similarly, there is a polynomial $G(t)\in\Z[t] $ such that
\begin{equation}\label{eq9922TCM.012d4}
G(\tau(p^{2n}))=\tau(p)G_0(\tau(p))\pm p^{11n}
\end{equation}for any integer $n\geq1$, so $G(t)=\pm p^{11n}$ at $t=\tau(p)=0$.
\end{proof}

%SSSSSSSSSSSSSSSSSSSSSSSSSSSSSSSSSSSSSSSSSSSSSSSSSSSSSSSSSSSSSSSSSSSSSSSSSSSSSSSSSSS
%SSSSSSSSSSSSSSSSSSSSSSSSSSSSSSSSSSSSSSSSSSSSSSSSSSSSSSSSSSSSSSSSSSSSSSSSSSSSSSSSSSS
%SSSSSSSSSSSSSSSSSSSSSSSSSSSSSSSSSSSSSSSSSSSSSSSSSSSSSSSSSSSSSSSSSSSSSSSSSSSSSSSSSSS
%SSSSSSSSSSSSSSSSSSSSSSSSSSSSSSSSSSSSSSSSSSSSSSSSSSSSSSSSSSSSSSSSSSSSSSSSSSSSSSSSSSS
%SSSSSSSSSSSSSSSSSSSSSSSSSSSSSSSSSSSSSSSSSSSSSSSSSSSSSSSSSSSSSSSSSSSSSSSSSSSSSSSSSSS
%SSSSSSSSSSSSSSSSSSSSSSSSSSSSSSSSSSSSSSSSSSSSSSSSSSSSSSSSSSSSSSSSSSSSSSSSSSSSSSSSSSS
%SSSSSSSSSSSSSSSSSSSSSSSSSSSSSSSSSSSSSSSSSSSSSSSSSSSSSSSSSSSSSSSSSSSSSSSSSSSSSSSSSSS
\section{Some Congruences}\label{S5522TC-C}\hypertarget{S5522TC-C}
The Fourier coefficients of some modular function satisfy some stringent and rigid congruences. A few congruences for the tau function are included here.  
\begin{lem}\label{lem5522TC.002}\hypertarget{lem5522TC.002} {\normalfont({\color{red}\cite[p. \ 24]{ZD2003}})} If $n \geq 1$ is an integer, then 
	\begin{equation} \label{eq5522TC.004}
		\tau(n)\equiv  
		\begin{cases}
			1 \bmod 2& \text { if } n=(2a+1)^{2b},\\
			0\bmod 2& \text { if } n\ne(2a+1)^{2b},
		\end{cases}
	\end{equation}
	where $a,b\in \N$.
\end{lem}
The sum of divisors function is defined by $\sigma_{s}(n)=\sum_{d\mid n}d^s$, with $s \in \C$. The Kolberg congruences express the values $\tau(n)$ at the odd integers $n\geq1$ as congruences of the sum of divisors function $\sigma_{11}(n)$. A general discussion of these congruences appears in {\color{red}\cite[p.\;4]{SP1972}}.
\begin{lem}\label{lem5522KC.008}\hypertarget{lem5522KC.008} {\normalfont({\color{red}\cite{KO1962}})} If $n \geq 1$ is an odd integer, then
	\begin{enumerate}[font=\normalfont, label=(\roman*)]
		\item $\displaystyle \tau(n)\equiv \sigma_{11}(n) \tmod 2^{11}$,    \tabto{10cm}if $n\equiv 1 \tmod 8$,
		\item $\displaystyle \tau(n)\equiv 1217\sigma_{11}(n) \tmod 2^{13}$,    \tabto{10cm}if $n\equiv 3 \tmod 8$,
		\item $\displaystyle \tau(n)\equiv 1537\sigma_{11}(n) \tmod 2^{12}$,    \tabto{10cm}if $n\equiv 5 \tmod 8$,
		\item $\displaystyle \tau(n)\equiv 705\sigma_{11}(n) \tmod 2^{14}$,    \tabto{10cm}if $n\equiv 7 \tmod 8$.
	\end{enumerate}
\end{lem}
%SSSSSSSSSSSSSSSSSSSSSSSSSSSSSSSSSSSSSSSSSSSSSSSSSSSSSSSSSSSSSSSSSSSSSSSSSSSSSSSSSSS
%SSSSSSSSSSSSSSSSSSSSSSSSSSSSSSSSSSSSSSSSSSSSSSSSSSSSSSSSSSSSSSSSSSSSSSSSSSSSSSSSSSS
%SSSSSSSSSSSSSSSSSSSSSSSSSSSSSSSSSSSSSSSSSSSSSSSSSSSSSSSSSSSSSSSSSSSSSSSSSSSSSSSSSSS
%SSSSSSSSSSSSSSSSSSSSSSSSSSSSSSSSSSSSSSSSSSSSSSSSSSSSSSSSSSSSSSSSSSSSSSSSSSSSSSSSSSS
%SSSSSSSSSSSSSSSSSSSSSSSSSSSSSSSSSSSSSSSSSSSSSSSSSSSSSSSSSSSSSSSSSSSSSSSSSSSSSSSSSSS
%SSSSSSSSSSSSSSSSSSSSSSSSSSSSSSSSSSSSSSSSSSSSSSSSSSSSSSSSSSSSSSSSSSSSSSSSSSSSSSSSSSS
%SSSSSSSSSSSSSSSSSSSSSSSSSSSSSSSSSSSSSSSSSSSSSSSSSSSSSSSSSSSSSSSSSSSSSSSSSSSSSSSSSSS
\section{Integers Representations and Quadratic Forms} \label{S1234IRQF-A}\hypertarget{S1234IRQF-A}
The quadratic form $ Q(x_1,...,x_m)\in \Z[x_1,...,x_m]$ is a linear combination
of $m$ squares.

\begin{thm}{\normalfont{\color{red}(\cite[Proposition 12]{DZ2013})} }\label{thm1234IRQF.100}\hypertarget{thm1234IRQF.100} Let $Q : \Z^m \longrightarrow \Z$ be a positive definite even unimodular
	quadratic form in $m$ variables. Then
	\begin{enumerate}[font=\normalfont, label=(\roman*)]
		\item the rank $m$ is divisible by $8$, and
		\item the number of representations of $n \in\N$ by $Q$ is given for large $n$ by the
		formula $$\displaystyle R_Q(n)=-\frac{2k}{B_k}\cdot \sigma_{k-1}(n)+O\left( n^{k/2}\right), $$
		where $m =2k$ and $B_k$ denotes the $k$th Bernoulli number.
	\end{enumerate}
\end{thm}

\begin{prop}\label{prop1234NRI.140}\hypertarget{prop1234NRI.140} Let $n\geq1$ be an integer and let $r_{24}(n)$ denotes the number of representations as sum of $24$ squares. Then
	\begin{equation}
		r_{24}(n)=\frac{16}{691}\sum_{d\mid n}(-1)^{n+d}d^{11}+\frac{128}{691}\left( (-1)^{n-1}259\tau(n)-512\tau(n/2)\right) .
	\end{equation}
\end{prop}
\begin{proof}[\textbf{Proof}]   This is a special case $s=24$ of a more general formula, see \cite{UW2014}, {\color{red}\cite[p.\;186]{MW2006}}, et alii.
\end{proof}

%SSSSSSSSSSSSSSSSSSSSSSSSSSSSSSSSSSSSSSSSSSSSSSSSSSSSSSSSSSSSSSSSSSSSSSSSSSSSSSSSSSS
%SSSSSSSSSSSSSSSSSSSSSSSSSSSSSSSSSSSSSSSSSSSSSSSSSSSSSSSSSSSSSSSSSSSSSSSSSSSSSSSSSSS
%SSSSSSSSSSSSSSSSSSSSSSSSSSSSSSSSSSSSSSSSSSSSSSSSSSSSSSSSSSSSSSSSSSSSSSSSSSSSSSSSSSS
%SSSSSSSSSSSSSSSSSSSSSSSSSSSSSSSSSSSSSSSSSSSSSSSSSSSSSSSSSSSSSSSSSSSSSSSSSSSSSSSSSSS
%SSSSSSSSSSSSSSSSSSSSSSSSSSSSSSSSSSSSSSSSSSSSSSSSSSSSSSSSSSSSSSSSSSSSSSSSSSSSSSSSSSS
%SSSSSSSSSSSSSSSSSSSSSSSSSSSSSSSSSSSSSSSSSSSSSSSSSSSSSSSSSSSSSSSSSSSSSSSSSSSSSSSSSSS
%SSSSSSSSSSSSSSSSSSSSSSSSSSSSSSSSSSSSSSSSSSSSSSSSSSSSSSSSSSSSSSSSSSSSSSSSSSSSSSSSSSS
\section{Partial Results for the Lehmer Conjecture}\label{S9988LC-S}\hypertarget{S9988LC-S}
The wide range of congruences available for the Ramanujan function are useful in establishing the nonvanishing property for many arithmetic progessions of primes, but not all. For example, the 3 Lemmas below work very well to exclude all the primes $p\equiv 1,3,5 \bmod 8$ but the same technique fails to establish that  $\tau(p)\ne0$ for $p\equiv 7 \bmod 8$. 
%LLLLLLLLLLLLLLLLLLLLLLLLLLLLLLLLLLLLLLLLLLLLLLLLLLLLLLLLLLLLLLLLLLLLLLLLLLLLLLL
%LLLLLLLLLLLLLLLLLLLLLLLLLLLLLLLLLLLLLLLLLLLLLLLLLLLLLLLLLLLLLLLLLLLLLLLLLLLLLLL
\begin{lem} \label{lem9988LC.092-1}\hypertarget{lem9988LC.092-1} If $p\equiv 1 \bmod 8$ is a prime, then $\tau(p)\ne0$.
\end{lem}

\begin{proof}[\textbf{Proof}]    To prove that $\tau(p)\ne0$ for all odd primes $p\equiv 1 \bmod 8$, the value of $\left |\tau(p^3)\right |$ is computed in two distinct ways. A comparison of these evaluations imply the claim.\\
	
	\textit{Evaluation} I. By \hyperlink{lem5522KC.008}{Lemma} \ref{lem5522KC.008}, 
	\begin{align}\label{eq9988LC.092d}
		\tau(n)&\equiv \sigma_{11}(n) \tmod 2^{11}\\
		&\equiv \sigma_{11}(n) \tmod 8\nonumber
	\end{align}
	for every odd integer $n\geq 1$. Letting $p\equiv 1\tmod 8$ and replacing $n=p^3$ yields
	\begin{eqnarray}\label{eq9988LC.092f}
		\tau(p^3)&\equiv &\sigma_{11}(p^3) \tmod 2^{8}\\
		&\equiv &\left( 1+p^{11}+p^{22}+p^{33}\right)  \tmod 8\nonumber\\
		&\equiv &4 \tmod 8\nonumber.
	\end{eqnarray}
	Therefore, there is an odd integer such that
	\begin{equation}\label{eq9988LC.092i}
		\tau(p^3)=4+8a_1,
	\end{equation}	
	where $a_1\in \Z$ is an integer, and 	
	the absolute value has the lower bound
	\begin{equation}\label{eq9988.092j}
		\left |\tau(p^3)\right |=\left |4+8a_1\right | \geq 4.
	\end{equation}
	
	\textit{Evaluation} II. By \hyperlink{lem9922TCM.012}{Lemma} \ref{lem9922TCM.012} or \hyperlink{exa9922TCM12.044d}{Example} \ref{exa9922TCM12.044d}, the function 
	\begin{equation}\label{eq9988.092i1}
		\tau(p^3)=\tau(p)^3-2p^{11}\tau(p)=\tau(p)\left ( \tau(p)^2-2p^{11}\right )
	\end{equation}
	is a polynomial in $\tau(p)$ and its absolute value is 
	\begin{equation}\label{eq9988.092i}
		|\tau(p^3) |=|\tau(p)|\cdot | \tau(p)^2-2p^{11} |.
	\end{equation}	
	. Comparing the absolute value of \eqref{eq9988.092i} and \eqref{eq9988.092j} returns 
	\begin{eqnarray}\label{eq9988.092k}
		\left |\tau(p)\right |\cdot \left |\tau(p)^2-2p^{11}\right |
		&=&\left |\tau(p^3)\right |\nonumber\\[.3cm]
		&=&\left |4+8a_1\right |\nonumber\\[.3cm]
		&\geq&4\nonumber.
	\end{eqnarray}
	Clearly, the case $\tau(p)= 0$ is false. Consequently, $\tau(p)\ne 0$ for all $p\equiv 1\tmod 8$.
\end{proof}
%LLLLLLLLLLLLLLLLLLLLLLLLLLLLLLLLLLLLLLLLLLLLLLLLLLLLLLLLLLLLLLLLLLLLLLLLLLLLLLL
%LLLLLLLLLLLLLLLLLLLLLLLLLLLLLLLLLLLLLLLLLLLLLLLLLLLLLLLLLLLLLLLLLLLLLLLLLLLLLLL
\begin{lem} \label{lem9988LC.092-3}\hypertarget{lem9988LC.092-3} If $p\equiv 3 \bmod 8$ is a prime, then $\tau(p)\ne0$.
\end{lem}

\begin{proof}[\textbf{Proof}]   To prove that $\tau(p)\ne0$ for all odd primes $p\equiv 3 \bmod 8$, the value of $\left |\tau(p^5)\right |$ is computed in two distinct ways. A comparison of these evaluations imply the claim.\\
	
	\textit{Evaluation} I. By \hyperlink{lem5522KC.008}{Lemma} \ref{lem5522KC.008}, 
	\begin{align}\label{eq9988LC.092d3}
		\tau(n)&\equiv 1217\cdot\sigma_{11}(n) \tmod 2^{13}\\
		&\equiv \sigma_{11}(n) \tmod 8\nonumber
	\end{align}
	for every odd integer $n\geq 1$. Letting $p\equiv 3\tmod 8$ and replacing $n=p^5$ yields
	\begin{eqnarray}\label{eq9988LC.092f3}
		\tau(p^5)&\equiv &\sigma_{11}(p^5) \tmod 8\\
		&\equiv &\left( 1+p^{11}+p^{22}+p^{33}+p^{44}+p^{55}\right)  \tmod 8\nonumber\\
		&\equiv &4 \tmod 8\nonumber.
	\end{eqnarray}
	Therefore, there is an odd integer such that
	\begin{equation}\label{eq9988LC.092i3}
		\tau(p^5)=4+8a_3,
	\end{equation}	
	where $a_3\in \Z$ is an integer, and 	
	the absolute value has the lower bound
	\begin{equation}\label{eq9988.092j3}
		\left |\tau(p^5)\right |=\left |4+8a_3\right | \geq 4.
	\end{equation}
	
	\textit{Evaluation} II. By \hyperlink{lem9922TCM.012}{Lemma} \ref{lem9922TCM.012} or \hyperlink{exa9922TCM12.044d}{Example} \ref{exa9922TCM12.044d}, the function 
	\begin{equation}\label{eq9988.092i13}
		\tau(p^5)=\tau(p)^5-4p^{11}\tau(p)^3+3p^{22}\tau(p)=\tau(p)\left (\tau(p)^4-4p^{11}\tau(p)^2+3p^{22}\right )
	\end{equation}
	is a polynomial in $\tau(p)$ and its absolute value is 
	\begin{equation}\label{eq9988.092i3}
		|\tau(p^5) |=|\tau(p)|\cdot | \tau(p)^4-4p^{11}\tau(p)^2+3p^{22} |.
	\end{equation}	
	. Comparing the absolute value of \eqref{eq9988.092i3} and \eqref{eq9988.092j3} returns 
	\begin{eqnarray}\label{eq9988.092k3}
		|\tau(p)|\cdot | \tau(p)^4-4p^{11}\tau(p)^2+3p^{22} |
		&=&\left |\tau(p^5)\right |\nonumber\\[.3cm]
		&=&\left |4+8a_3\right |\nonumber\\[.3cm]
		&\geq&4\nonumber.
	\end{eqnarray}
	Clearly, the case $\tau(p)= 0$ is false. Consequently, $\tau(p)\ne 0$ for all $p\equiv 3\tmod 8$.
\end{proof}

%LLLLLLLLLLLLLLLLLLLLLLLLLLLLLLLLLLLLLLLLLLLLLLLLLLLLLLLLLLLLLLLLLLLLLLLLLLLLLLL
%LLLLLLLLLLLLLLLLLLLLLLLLLLLLLLLLLLLLLLLLLLLLLLLLLLLLLLLLLLLLLLLLLLLLLLLLLLLLLLL
\begin{lem} \label{lem9988LC.092-5}\hypertarget{lem9988LC.092-5} If $p\equiv 5 \bmod 8$ is a prime, then $\tau(p)\ne0$.
\end{lem}

\begin{proof}[\textbf{Proof}]    To prove that $\tau(p)\ne0$ for all odd primes $p\equiv 5 \bmod 8$, the value of $\left |\tau(p^3)\right |$ is computed in two distinct ways. A comparison of these evaluations imply the claim.\\
	
	\textit{Evaluation} I. By \hyperlink{lem5522KC.008}{Lemma} \ref{lem5522KC.008}, 
	\begin{align}\label{eq9988LC.092d5}
		\tau(n)&\equiv 1537\cdot\sigma_{11}(n) \tmod 2^{13}\\
		&\equiv \sigma_{11}(n) \tmod 8\nonumber
	\end{align}
	for every odd integer $n\geq 1$. Letting $p\equiv 5\tmod 8$ and replacing $n=p^3$ yields
	\begin{eqnarray}\label{eq9988LC.092f5}
		\tau(p^3)&\equiv &\sigma_{11}(p^3) \tmod 2^{8}\\
		&\equiv &\left( 1+p^{11}+p^{22}+p^{33}\right)  \tmod 8\nonumber\\
		&\equiv &2 \tmod 8\nonumber.
	\end{eqnarray}
	Therefore, there is an odd integer such that
	\begin{equation}\label{eq9988LC.092i5}
		\tau(p^3)=2+8a_5,
	\end{equation}	
	where $a_5\in \Z$ is an integer, and 	
	the absolute value has the lower bound
	\begin{equation}\label{eq9988.092j5}
		\left |\tau(p^3)\right |=\left |2+8a_5\right | \geq 2.
	\end{equation}
	
	\textit{Evaluation} II. By \hyperlink{lem9922TCM.012}{Lemma} \ref{lem9922TCM.012} or \hyperlink{exa9922TCM12.044d}{Example} \ref{exa9922TCM12.044d}, the function 
	\begin{equation}\label{eq9988.092i9}
		\tau(p^3)=\tau(p)^3-2p^{11}\tau(p)=\tau(p)\left ( \tau(p)^2-2p^{11}\right )
	\end{equation}
	is a polynomial in $\tau(p)$ and its absolute value is 
	\begin{equation}\label{eq9988.092i5}
		|\tau(p^3) |=|\tau(p)|\cdot | \tau(p)^2-2p^{11} |.
	\end{equation}	
	Comparing the absolute value of \eqref{eq9988.092i5} and \eqref{eq9988.092j5} returns 
	\begin{eqnarray}\label{eq9988.092k5}
		\left |\tau(p)\right |\cdot \left |\tau(p)^2-2p^{11}\right |
		&=&\left |\tau(p^3)\right |\nonumber\\[.3cm]
		&=&\left |2+8a_5\right |\nonumber\\[.3cm]
		&\geq&2\nonumber.
	\end{eqnarray}
	Clearly, the case $\tau(p)= 0$ is false. Consequently, $\tau(p)\ne 0$ for all $p\equiv 5\tmod 8$.
\end{proof}

%SSSSSSSSSSSSSSSSSSSSSSSSSSSSSSSSSSSSSSSSSSSSSSSSSSSSSSSSSSSSSSSSSSSSSSSSSSSSSSSSSSS
%SSSSSSSSSSSSSSSSSSSSSSSSSSSSSSSSSSSSSSSSSSSSSSSSSSSSSSSSSSSSSSSSSSSSSSSSSSSSSSSSSSS
%SSSSSSSSSSSSSSSSSSSSSSSSSSSSSSSSSSSSSSSSSSSSSSSSSSSSSSSSSSSSSSSSSSSSSSSSSSSSSSSSSSS
%SSSSSSSSSSSSSSSSSSSSSSSSSSSSSSSSSSSSSSSSSSSSSSSSSSSSSSSSSSSSSSSSSSSSSSSSSSSSSSSSSSS
%SSSSSSSSSSSSSSSSSSSSSSSSSSSSSSSSSSSSSSSSSSSSSSSSSSSSSSSSSSSSSSSSSSSSSSSSSSSSSSSSSSS
%SSSSSSSSSSSSSSSSSSSSSSSSSSSSSSSSSSSSSSSSSSSSSSSSSSSSSSSSSSSSSSSSSSSSSSSSSSSSSSSSSSS
%SSSSSSSSSSSSSSSSSSSSSSSSSSSSSSSSSSSSSSSSSSSSSSSSSSSSSSSSSSSSSSSSSSSSSSSSSSSSSSSSSSS
\section{Results for the Lehmer Conjecture}\label{S9988LC-W}\hypertarget{S9988LC-W}
The vanishing of some coefficients of the product
\begin{align}\label{eq5599NTF.500W00}
\sum_{n\geq1}a(n)x^{n}=\prod_{n\geq1}\left( 1-x^{n}\right)^m 
\end{align}
for some parameters $m\leq15$ is discussed by Lehmer and asks the about the situation for $m=24$. Specifically, $\Delta(x)=\sum_{n\geq1}\tau(n)x^{n-1}=x\prod_{n\geq}\left(1-x^n \right)^{24}$. 

\begin{conj}{\normalfont ({\color{red}\cite[p.\;429]{LD1947}}) }\label{conj9988LC.500W}\hypertarget{conj9988LC.500W} Is the Ramanujan function $\tau(n)= 0$ for any integers $n\in\N$.
\end{conj}

Some of the reported numerical calculations are the followings.
\begin{enumerate}
	\item $\tau(p)\ne 0$ for $p<3316798$. In \cite{LD1947}, Lehmer used fast primality test to verify this inequality.
	\item $\tau(p)\ne 0$ for $p<10^{15}$. In \cite{SJ2003}, Serre used congruences to verify this inequality. 
	\item $\tau(n)\ne 0$ for $n<816212624008487344127999\approx 8\times 10^{23}$.  In {\color{red}\cite[Corollary 1.2.]{DZ2013}}, Derickx et alii, used Galois representation theory and other advanced techniques to verify the Fourier coefficient inequality $\tau(n)\ne 0$ up to this value.
\end{enumerate}
To investigate the claim that $\tau(n)\ne 0$ for $n \geq 1$, it is sufficient to consider $\tau(p)\ne 0$ for prime $p\geq 2$. This is immediately implies by the multiplicative properties of the Fourier coefficients in \hyperlink{lem9922TCM.012}{Lemma} \ref{lem9922TCM.012}. The previous argument for this reduction as in {\color{red}\cite[Theorem 2]{LD1947}} is more complicated.

\begin{thm}\label{thm9988LC.500W}\hypertarget{thm9988LC.500W}The $n$th coefficient  $\tau(n)\ne 0$ of the discriminat function $\Delta(q)=\sum_{n\geq1}\tau(n)q^n$ does not vanishes for all integers $n\in\N$.
\end{thm}

\begin{proof}[\textbf{Proof}]  
Suppose that there exists a large prime $p\geq p_0>8\times 10^{23}$ such that
\begin{align}\label{eq5599NTF.500W02}
	\tau(p)&=0. 
\end{align}
Next, observe that for any prime $p$, the coefficient $\tau(p^6)$ has an expansion as a polynomial in $p$ and $\tau(p)$:
\begin{align}\label{eq5599NTF.500W04}
	\tau(p^6)&= \tau(p)^6-5p^{11}\tau(p)^4 +6p^{11}\tau(p)^2-p^{33}, 
\end{align}
see \hyperlink{exa9922TCM12.044d}{Example }\ref{exa9922TCM12.044d}. Consequently, the hypothesis \eqref{eq5599NTF.500W02} and \eqref{eq5599NTF.500W04} imply that \begin{align}\label{eq5599NTF.500W06}
	\tau(p^6)&= -p^{33}. 
\end{align}
Now, evaluating the formula for the number of representations of $p^6$ as sum of 24 squares, see \hyperlink{prop1234NRI.140}{Proposition} \ref{prop1234NRI.140}, yields
\begin{align}\label{eq5599NTF.500W08}
	r_{24}(n)&=\frac{16}{691}\sum_{d\mid n}(-1)^{n+d}d^{11} +\frac{128}{691}\left((-1)^{n-1}259\tau(n)-512\tau(n/2) \right)\\[.3cm]
	r_{24}(p^6)&=\frac{16}{691}\left( 1+p^{11}+p^{22}+p^{33}+p^{44}+p^{55}+p^{66}\right)+\frac{128}{691}\left(259\tau(p^6)-512\tau(p^6/2) \right)\nonumber\\[.3cm]
	&=\frac{16}{691}\left( 1+p^{11}+p^{22}+p^{33}+p^{44}+p^{55}+p^{66}\right)-\frac{128\cdot 259}{691}p^{33},\nonumber
\end{align}
where $\tau(n/2)=\tau(p^6/2)=0$ for any odd integer $n=p^6$. Rearranging \eqref{eq5599NTF.500W08} yields
\begin{align}\label{eq5599NTF.500W12}
	\frac{691}{16}\cdot r_{24}(p^6)&=\left( 1+p^{11}+p^{22}+p^{33}+p^{44}+p^{55}+p^{66}\right)-8\cdot 259p^{33}\\[.3cm]
	&=1+p^{11}+p^{22}-2071p^{33}+p^{44}+p^{55}+p^{66}\nonumber.
\end{align}
Replacing $x=p^{11}$ and reducing it modulo 691 yield the polynomial equation
\begin{align}\label{eq5599NTF.500W14}
	1+x+x^{2}-674x^{3}+x^{4}+x^{5}+x^{6}\equiv 0\bmod 691.
\end{align}
An exustive search or a computer algebra system immediate shows that the polynomial equation has no solution in $\F_{691}$. Hence, it follows that the hypothetical prime $p>8\times 10^{23}$ in hypothesis \eqref{eq5599NTF.500W02} does not exists.
\end{proof}

This result has intersting application to the arithmetic properties of the coefficients of some mock modular forms and to spherical $t$-design theory, but these topics are beyond the level of this work. 
%SSSSSSSSSSSSSSSSSSSSSSSSSSSSSSSSSSSSSSSSSSSSSSSSSSSSSSSSSSSSSSSSSSSSSSSSSSSSSSSSSSS
%SSSSSSSSSSSSSSSSSSSSSSSSSSSSSSSSSSSSSSSSSSSSSSSSSSSSSSSSSSSSSSSSSSSSSSSSSSSSSSSSSSS
%SSSSSSSSSSSSSSSSSSSSSSSSSSSSSSSSSSSSSSSSSSSSSSSSSSSSSSSSSSSSSSSSSSSSSSSSSSSSSSSSSSS
%SSSSSSSSSSSSSSSSSSSSSSSSSSSSSSSSSSSSSSSSSSSSSSSSSSSSSSSSSSSSSSSSSSSSSSSSSSSSSSSSSSS
%SSSSSSSSSSSSSSSSSSSSSSSSSSSSSSSSSSSSSSSSSSSSSSSSSSSSSSSSSSSSSSSSSSSSSSSSSSSSSSSSSSS
%SSSSSSSSSSSSSSSSSSSSSSSSSSSSSSSSSSSSSSSSSSSSSSSSSSSSSSSSSSSSSSSSSSSSSSSSSSSSSSSSSSS
%SSSSSSSSSSSSSSSSSSSSSSSSSSSSSSSSSSSSSSSSSSSSSSSSSSSSSSSSSSSSSSSSSSSSSSSSSSSSSSSSSSS
\section{Finite Sum of Fourier Coefficients} \label{S8285STF-S}\hypertarget{S8285TSTF-S}
The summatory function of the Fourier coefficients
of cusp form $f(s)=\sum_{n\geq1}a(n)n^{-s}$ satisfies 
\begin{equation}\label{eq8285STF.200a02}
	\sum_{n \leq x} a(n)=O \left (x^{(k-1)/2+1/3} \right ),
\end{equation}
the proof appears in \cite{HJ1989}. In comparison, the classical conjecture states that
\begin{equation}\label{eq8285STF.200a04}
	\sum_{n \leq x} a(n)=O \left (x^{(k-1)/2+1/4+\epsilon} \right ),
\end{equation}

%SSSSSSSSSSSSSSSSSSSSSSSSSSSSSSSSSSSSSSSSSSSSSSSSSSSSSSSSSSSSSSSSSSSSSSSSSSSSSSSSSSS
%SSSSSSSSSSSSSSSSSSSSSSSSSSSSSSSSSSSSSSSSSSSSSSSSSSSSSSSSSSSSSSSSSSSSSSSSSSSSSSSSSSS
%SSSSSSSSSSSSSSSSSSSSSSSSSSSSSSSSSSSSSSSSSSSSSSSSSSSSSSSSSSSSSSSSSSSSSSSSSSSSSSSSSSS
%SSSSSSSSSSSSSSSSSSSSSSSSSSSSSSSSSSSSSSSSSSSSSSSSSSSSSSSSSSSSSSSSSSSSSSSSSSSSSSSSSSS
%SSSSSSSSSSSSSSSSSSSSSSSSSSSSSSSSSSSSSSSSSSSSSSSSSSSSSSSSSSSSSSSSSSSSSSSSSSSSSSSSSSS
%SSSSSSSSSSSSSSSSSSSSSSSSSSSSSSSSSSSSSSSSSSSSSSSSSSSSSSSSSSSSSSSSSSSSSSSSSSSSSSSSSSS
%SSSSSSSSSSSSSSSSSSSSSSSSSSSSSSSSSSSSSSSSSSSSSSSSSSSSSSSSSSSSSSSSSSSSSSSSSSSSSSSSSSS
\section{Average Order of Sum of Absolute Tau Function}\label{S8282TSTF-C}\hypertarget{S8282TSTF-C}
The \textit{Mobius randomness principle} anticipates a sizable cancellations for any twisted finite sum $\sum_{n\le x}\mu(n)f(n)$. The actual cancellation is determined by the zerofree region and weight $k\geq2$ of the modular form $f(s)=\sum_{n\geq1}a(n)n^{-s}$. The result below attempts to compute the average order for the twisted finite sum with $f(n)=\tau(n)$.

\begin{prop} \label{prop8282TSTF.012C}\hypertarget{prop8282TSTF.012C} If $x>1$ is a large real number, then
	\begin{equation}\sum_{n \leq x} | \tau(n)|=O \left (x^{13/2} \log x \right ).
	\end{equation}
\end{prop}

\begin{proof}[\textbf{Proof}]  (i) By Deligne theorem, $|\tau(n)\leq d(n)n^{11/2}$. Let $A(t)=\sum_{n\leq t}d(n)=x\log x+O(x)$. Summing this inequality over the interval $[1,x]$ yields
	\begin{eqnarray}
		\sum_{n \leq x} | \tau(n)| &\leq & \sum_{n \leq x} d(n)n^{11/2}? \\[.3cm]
		& \leq&\int_1^x t^{11/2}d\;A(t)\nonumber\\[.3cm]
		&\ll& x^{13/2}\log x\nonumber.
	\end{eqnarray}
\end{proof}
%SSSSSSSSSSSSSSSSSSSSSSSSSSSSSSSSSSSSSSSSSSSSSSSSSSSSSSSSSSSSSSSSSSSSSSSSSSSSSSSSSSS
%SSSSSSSSSSSSSSSSSSSSSSSSSSSSSSSSSSSSSSSSSSSSSSSSSSSSSSSSSSSSSSSSSSSSSSSSSSSSSSSSSSS
%SSSSSSSSSSSSSSSSSSSSSSSSSSSSSSSSSSSSSSSSSSSSSSSSSSSSSSSSSSSSSSSSSSSSSSSSSSSSSSSSSSS
%SSSSSSSSSSSSSSSSSSSSSSSSSSSSSSSSSSSSSSSSSSSSSSSSSSSSSSSSSSSSSSSSSSSSSSSSSSSSSSSSSSS
%SSSSSSSSSSSSSSSSSSSSSSSSSSSSSSSSSSSSSSSSSSSSSSSSSSSSSSSSSSSSSSSSSSSSSSSSSSSSSSSSSSS
%SSSSSSSSSSSSSSSSSSSSSSSSSSSSSSSSSSSSSSSSSSSSSSSSSSSSSSSSSSSSSSSSSSSSSSSSSSSSSSSSSSS
%SSSSSSSSSSSSSSSSSSSSSSSSSSSSSSSSSSSSSSSSSSSSSSSSSSSSSSSSSSSSSSSSSSSSSSSSSSSSSSSSSSS
\section{Twisted Sum of Tau Function}\label{S8282TSTF-S}\hypertarget{S8282TSTF-S}
The \textit{Mobius randomness principle} anticipates a sizable cancellations for any twisted finite sum $\sum_{n\le x}\mu(n)f(n)$. The actual cancellation is determined by the zerofree region and weight $k\geq2$ of the modular form $f(s)=\sum_{n\geq1}a(n)n^{-s}$. The result below attempts to compute the average order for the twisted finite sum with $f(n)=\tau(n)$.

\begin{prop} \label{prop8282TSTF.012}\hypertarget{prop8282TSTF.012} If $x>1$ is a large real number, then
	\begin{enumerate}[font=\normalfont, label=(\roman*)]
		\item $\displaystyle \sum_{n \leq x} \mu(n) \tau(n)=O \left (x^{6} \log^{2}x \right )$,\tabto{7cm}unconditionally.
		\item $\displaystyle \sum_{n \leq x} \mu(n) \tau(n)=O \left (x^{6-1/4} \log^{2}x \right )$, \tabto{7cm}condtional on the generalized RH.
	\end{enumerate}
\end{prop}

\begin{proof}[\textbf{Proof}]   First derive some form of representative for the generating funcction:
\begin{eqnarray}
	\sum_{n \geq 1} \frac{\mu(n) \tau(n)}{n^s} &=& \prod_{p \geq 2} \left ( 1- \frac{\tau(p)}{p^s} \right ) \\[.3cm]
	& =&\frac{1}{L(s,\tau)}   \prod_{p \geq 2} \left ( 1- \frac{\tau(p)}{p^s} \right )  \left ( \frac{1}{p^{11-2s}-\tau(p)p^{-s}+1} \right )\nonumber\\[.3cm]
	&=& \frac{1}{L(s,\tau)}   \prod_{p \geq 2} \left ( 1-  \frac{p^{11-s}}{p^{11-s}-\tau(p)+p^{s}} \right ) \nonumber\\[.3cm]
	&=& \frac{1}{L(s,\tau)}  \prod_{p \geq 2} \left ( 1-  \frac{1}{p^{2s-11}-\tau(p)p^{s-1}+1} \right )\nonumber\\[.3cm]
	&=& \frac{g(s)}{L(s,\tau)}\nonumber.
\end{eqnarray}

Since the last product, denoted by $g(s)$, is absolutely convergent for $\Re e(s)>(k-1)/2$, it is a holomorphic function of $s \in \mathbb{C}.$ The $L$-function attached to a cusp form of weight $k \geq 1$,  

\begin{equation}
	L(s, \gamma)=\sum_{n \geq 1} \frac{a(n)}{n^s}
\end{equation}
is uniformly and absolutely convergent on the complex half plane $ \Re e(s)>k/2+1$, see {\color{red}\cite[p. \;260]{ZD1999}}. An application of the Perron formula completes the proof.
\end{proof}

%SSSSSSSSSSSSSSSSSSSSSSSSSSSSSSSSSSSSSSSSSSSSSSSSSSSSSSSSSSSSSSSSSSSSSSSSSSSSSSSSSSS
%SSSSSSSSSSSSSSSSSSSSSSSSSSSSSSSSSSSSSSSSSSSSSSSSSSSSSSSSSSSSSSSSSSSSSSSSSSSSSSSSSSS
%SSSSSSSSSSSSSSSSSSSSSSSSSSSSSSSSSSSSSSSSSSSSSSSSSSSSSSSSSSSSSSSSSSSSSSSSSSSSSSSSSSS
%SSSSSSSSSSSSSSSSSSSSSSSSSSSSSSSSSSSSSSSSSSSSSSSSSSSSSSSSSSSSSSSSSSSSSSSSSSSSSSSSSSS
%SSSSSSSSSSSSSSSSSSSSSSSSSSSSSSSSSSSSSSSSSSSSSSSSSSSSSSSSSSSSSSSSSSSSSSSSSSSSSSSSSSS
%SSSSSSSSSSSSSSSSSSSSSSSSSSSSSSSSSSSSSSSSSSSSSSSSSSSSSSSSSSSSSSSSSSSSSSSSSSSSSSSSSSS
%SSSSSSSSSSSSSSSSSSSSSSSSSSSSSSSSSSSSSSSSSSSSSSSSSSSSSSSSSSSSSSSSSSSSSSSSSSSSSSSSSSS
\section{Estimate for the Vanishing Set of Primes} \label{S9810PVC}\hypertarget{S9810PVC}
The hypothetical subset of primes on which the tau function vanishes can estimate both unconditionally and conditionally.

\begin{prop}\label{prop9810PVC.300}\hypertarget{prop9810PVC.300} If $x$ is a large nonnegative real number and $\pi_{\tau}(x)=\#\{p\leq x: \tau(p)=0\}$ then
	\begin{enumerate}[font=\normalfont, label=(\roman*)]
	\item $\displaystyle \pi_{\tau}(x)\ll x(\log x)^{3/2}$,\tabto{7cm}unconditionally.
	\item $\displaystyle \pi_{\tau}(x)\ll  x^{3/4}$, \tabto{7cm}condtional on the generalized RH.
\end{enumerate}
\end{prop}	
\begin{proof}[\textbf{Proof}]   This result appears in \cite{SJ2003}.
\end{proof}

%SSSSSSSSSSSSSSSSSSSSSSSSSSSSSSSSSSSSSSSSSSSSSSSSSSSSSSSSSSSSSSSSSSSSSSSSSSSSSSSSSSS
%SSSSSSSSSSSSSSSSSSSSSSSSSSSSSSSSSSSSSSSSSSSSSSSSSSSSSSSSSSSSSSSSSSSSSSSSSSSSSSSSSSS
%SSSSSSSSSSSSSSSSSSSSSSSSSSSSSSSSSSSSSSSSSSSSSSSSSSSSSSSSSSSSSSSSSSSSSSSSSSSSSSSSSSS
%SSSSSSSSSSSSSSSSSSSSSSSSSSSSSSSSSSSSSSSSSSSSSSSSSSSSSSSSSSSSSSSSSSSSSSSSSSSSSSSSSSS
%SSSSSSSSSSSSSSSSSSSSSSSSSSSSSSSSSSSSSSSSSSSSSSSSSSSSSSSSSSSSSSSSSSSSSSSSSSSSSSSSSSS
%SSSSSSSSSSSSSSSSSSSSSSSSSSSSSSSSSSSSSSSSSSSSSSSSSSSSSSSSSSSSSSSSSSSSSSSSSSSSSSSSSSS
%SSSSSSSSSSSSSSSSSSSSSSSSSSSSSSSSSSSSSSSSSSSSSSSSSSSSSSSSSSSSSSSSSSSSSSSSSSSSSSSSSSS
\section{Convolutions of Sum of Divisors Functions} \label{S9810DFC}\hypertarget{S9810DFC}
For a pair of integers $n\geq1$ and $a\geq 0$ the sum of divisors function is defined by
\begin{equation}\label{eq9810.600d}
	\sigma_s(n)=\sum_{d\mid n} d^{s}.
\end{equation}	
For an integer parameter $r\geq 0$ the convolution sum is defined by
\begin{equation}\label{eq9810.600f}
	S_a(r,n)=\sum_{0< k< n} k^{r}\sigma_a(k)\sigma_a(n-k).
\end{equation}	
The corresponding recursive formula is given by the finite sum
\begin{equation}\label{eq9810.600e}
	S_a(r,n)=\sum_{0\leq k\leq r} (-1)^{k}n^{r-k}\binom{r}{k}S_a(k,n).
\end{equation}	

\begin{eqnarray}\label{eq9810.600g}
	S_1(4,n)&=&\sum_{0\leq k\leq 4} (-1)^{k}n^{4-k}\binom{4}{k}S_1(k,n)\\
	&=&n^4S_1(0,n)-4n^3S_1(3,n)+6n^2S_1(2,n)-4nS_1(4,n)	\nonumber\\.
\end{eqnarray}

\begin{lem} \label{lem9810.600} \hypertarget{lem9810.600}{\normalfont(Niebur formula)} Let $n\geq1$ be an integer and let $\tau:\N\longrightarrow \Z$ be the tau function. Then 
	\begin{align}\label{eq9810.600i}
		\tau(n) &= n^4\sigma(n)-24\left( 35S_1(4,n)-52nS_1(3,n)+18n^2S_1(2,n)\right)\\[.2cm]
		&= n^4\sigma(n)-24\sum_{1\leq k\leq n-1}\left( 35k^4-52k^3n+18k^2n^2\right)\sigma(k)\sigma(n-k).
	\end{align}	
\end{lem}
\begin{proof}[\textbf{Proof}]   A complete proof and other references appears in {\color{red}\cite[Section 3]{GL2010}}.
\end{proof}

\begin{lem} \label{lem9810.700} \hypertarget{lem9810.700}{\normalfont(Convolution formulas)} Let $p\geq2$ be a prime number and let $\tau:\N\longrightarrow \Z$ be the tau function. Then 
	\begin{enumerate}[font=\normalfont, label=(\roman*)]
		\item 
		$\displaystyle S_1(0, p) = \frac{1}{12}(p-1)(5p-6)(p + 1)$.
		
		\item 
		$\displaystyle S_1(1, p) = \frac{1}{24}(p-1)(5p-6)(p + 1)p$.
		
		\item 
		$\displaystyle S_1(2, p) = \frac{1}{24}(p-1)(3p-4)(p + 1)p^2$. 
		
		\item 
		$\displaystyle S_1(3, p) = \frac{1}{24}(p-1)(2p-3)(p + 1)p^3$. 
		
		\item 
		$\displaystyle S_1(4, p) = \frac{1}{840}\left( (50p^2 - 134p + 85)(p + 1)p^4-\tau(p)\right) $. 
	\end{enumerate}
\end{lem}

\begin{proof}[\textbf{Proof}]   The complete proofs and references appears in \cite{GL2010}.
\end{proof}

%SSSSSSSSSSSSSSSSSSSSSSSSSSSSSSSSSSSSSSSSSSSSSSSSSSSSSSSSSSSSSSSSSSSSSSSSSSSSSSSSSSS
%SSSSSSSSSSSSSSSSSSSSSSSSSSSSSSSSSSSSSSSSSSSSSSSSSSSSSSSSSSSSSSSSSSSSSSSSSSSSSSSSSSS
%SSSSSSSSSSSSSSSSSSSSSSSSSSSSSSSSSSSSSSSSSSSSSSSSSSSSSSSSSSSSSSSSSSSSSSSSSSSSSSSSSSS
%SSSSSSSSSSSSSSSSSSSSSSSSSSSSSSSSSSSSSSSSSSSSSSSSSSSSSSSSSSSSSSSSSSSSSSSSSSSSSSSSSSS
%SSSSSSSSSSSSSSSSSSSSSSSSSSSSSSSSSSSSSSSSSSSSSSSSSSSSSSSSSSSSSSSSSSSSSSSSSSSSSSSSSSS
%SSSSSSSSSSSSSSSSSSSSSSSSSSSSSSSSSSSSSSSSSSSSSSSSSSSSSSSSSSSSSSSSSSSSSSSSSSSSSSSSSSS
%SSSSSSSSSSSSSSSSSSSSSSSSSSSSSSSSSSSSSSSSSSSSSSSSSSSSSSSSSSSSSSSSSSSSSSSSSSSSSSSSSSS
\section{Estimates for the Moments of Convolutions}
\label{S9090MCDF-S}\hypertarget{S9090MCDF-S}
The $m$th moment of the convolution of the sum of divisors function is defined by
\begin{equation}\label{eq9090MCDF.111f2}
\sum_{1\leq k\leq n-1}k^m\sigma(k)\sigma(n-k).
\end{equation}	
The basic details for a pair of basic estimates of the $m$th moment are outlined here. 
%LLLLLLLLLLLLLLLLLLLLLLLLLLLLLLLLLLLLLLLLLLLLLLLLLLLLLLLLLLLLLLLLLLLLLLLLLLL
%LLLLLLLLLLLLLLLLLLLLLLLLLLLLLLLLLLLLLLLLLLLLLLLLLLLLLLLLLLLLLLLLLLLLLLLLLL
\begin{lem} \label{lem9090MCDF.111-D}\hypertarget{lem9090MCDF.111-D} If $m,n\geq 1$ are a pair of integers, then 

\begin{enumerate}[font=\normalfont, label=(\roman*)]
\item $\displaystyle \sum_{1\leq k\leq n-1}k\sigma(k)\sigma(n-k)\geq \frac{n^4}{12}-\frac{n^2}{4}+\frac{n}{6}$,
\item $\displaystyle \sum_{1\leq k\leq n-1}k^2\sigma(k)\sigma(n-k)\geq \frac{n^5}{20}-\frac{n^3}{12}+\frac{n}{5}$,
\item $\displaystyle \sum_{1\leq k\leq n-1}k^3\sigma(k)\sigma(n-k)\geq  \frac{n^6}{30}-\frac{n^4}{12}+\frac{n^2}{20}$,
\item $\displaystyle \sum_{1\leq k\leq n-1}k^4\sigma(k)\sigma(n-k)\geq\frac{n^7}{42}-\frac{n^5}{6}+\frac{n^4}{3}+\frac{2n^3}{15}-\frac{n}{42}$.
\end{enumerate}
\end{lem}

\begin{proof}[\textbf{Proof}]  (i) The sum of divisors function has the trivial lower bound $\sigma(k)\geq k$, so 
\begin{align}\label{eq9090MCDF.111f6}
\sum_{1\leq k\leq n-1}k\sigma(k)\sigma(n-k)&\geq\sum_{1\leq k\leq n-1}k\cdot  k\cdot (n-k)\\[.3cm]
&\geq n\sum_{1\leq k\leq n-1}k^2-\sum_{1\leq k\leq n-1}k^3 \nonumber\\[.3cm]
		&\geq n\left( \frac{n^3}{3}+\frac{n^2}{2}+\frac{n}{6}\right) -\left( \frac{n^4}{4}+\frac{n^3}{2}+\frac{n^2}{4}\right) \nonumber\\[.3cm].
		&\geq \frac{n^4}{12}-\frac{n^2}{4}+\frac{n}{6} \nonumber.
	\end{align}	
(ii) The sum of divisors function has the trivial lower bound $\sigma(k)\geq k$, so 
\begin{align}\label{eq9090MCDF.111f10}
	\sum_{1\leq k\leq n-1}k^2\sigma(k)\sigma(n-k)&\geq\sum_{1\leq k\leq n-1}k^2\cdot  k\cdot (n-k)\\[.3cm]
	&\geq n\sum_{1\leq k\leq n-1}k^3-\sum_{1\leq k\leq n-1}k^4 \\[.3cm]
	&\geq n\left( \frac{n^4}{4}+\frac{n^3}{2}+\frac{n^2}{4}\right) -\left( \frac{n^5}{5}+\frac{n^4}{2}+\frac{n^3}{3}-\frac{n}{30}\right) \nonumber\\[.3cm].
	&\geq \frac{n^5}{20}-\frac{n^3}{12}+\frac{n}{5} \nonumber.
\end{align}	
(iii) The sum of divisors function has the trivial lower bound $\sigma(k)\geq k$, so 
\begin{align}\label{eq9090MCDF.111f13}
	\sum_{1\leq k\leq n-1}k^3\sigma(k)\sigma(n-k)&\geq\sum_{1\leq k\leq n-1}k^3\cdot  k\cdot (n-k)\\[.3cm]
	&\geq n\sum_{1\leq k\leq n-1}k^4-\sum_{1\leq k\leq n-1}k^5 \\[.3cm]
	&\geq n\left( \frac{n^5}{5}+\frac{n^4}{2}+\frac{n^3}{3}-\frac{n}{30}\right) -\left( \frac{n^6}{6}+\frac{n^5}{2}+\frac{5n^4}{12}-\frac{n^2}{12}\right) \nonumber\\[.3cm].
	&\geq \frac{n^6}{30}-\frac{n^4}{12}+\frac{n^2}{20} \nonumber.
\end{align}	
(iv) The sum of divisors function has the trivial lower bound $\sigma(k)\geq k$, so 
\begin{align}\label{eq9090MCDF.111f15}
	\sum_{1\leq k\leq n-1}k^4\sigma(k)\sigma(n-k)&\geq\sum_{1\leq k\leq n-1}k^4\cdot  k\cdot (n-k)\\[.3cm]
	&\geq n\sum_{1\leq k\leq n-1}k^5-\sum_{1\leq k\leq n-1}k^6 \\[.3cm]
	&\geq n\left( \frac{n^6}{6}+\frac{n^5}{2}+\frac{n^4}{3}-\frac{n^2}{30}\right) -\left( \frac{n^7}{7}+\frac{n^6}{2}+\frac{n^5}{2}-\frac{n^3}{6}+\frac{n}{42}\right) \nonumber\\[.3cm].
	&\geq \frac{n^7}{42}-\frac{n^5}{6}+\frac{n^4}{3}+\frac{2n^3}{15}-\frac{n}{42} \nonumber.
\end{align}	
		
\end{proof}

%SSSSSSSSSSSSSSSSSSSSSSSSSSSSSSSSSSSSSSSSSSSSSSSSSSSSSSSSSSSSSSSSSSSSSSSSSSSSSSSSSSS
%SSSSSSSSSSSSSSSSSSSSSSSSSSSSSSSSSSSSSSSSSSSSSSSSSSSSSSSSSSSSSSSSSSSSSSSSSSSSSSSSSSS
%SSSSSSSSSSSSSSSSSSSSSSSSSSSSSSSSSSSSSSSSSSSSSSSSSSSSSSSSSSSSSSSSSSSSSSSSSSSSSSSSSSS
%SSSSSSSSSSSSSSSSSSSSSSSSSSSSSSSSSSSSSSSSSSSSSSSSSSSSSSSSSSSSSSSSSSSSSSSSSSSSSSSSSSS
%SSSSSSSSSSSSSSSSSSSSSSSSSSSSSSSSSSSSSSSSSSSSSSSSSSSSSSSSSSSSSSSSSSSSSSSSSSSSSSSSSSS
%SSSSSSSSSSSSSSSSSSSSSSSSSSSSSSSSSSSSSSSSSSSSSSSSSSSSSSSSSSSSSSSSSSSSSSSSSSSSSSSSSSS
%SSSSSSSSSSSSSSSSSSSSSSSSSSSSSSSSSSSSSSSSSSSSSSSSSSSSSSSSSSSSSSSSSSSSSSSSSSSSSSSSSSS
\section{Numerical Calculations of the Tau Function}\label{S1234NCT-S}\hypertarget{S1234NCT-S}
The calculations of a few small values of $\tau(n)$ using the divisor convolution functions are demeonstred below.

%TTTTTTTTTTTTTTTTTTTTTTTTTTTTTTTTTTTTTTTTTTTTTTTTTTTTTTTTTTTTTTTTTTTTTTTTTTTTT
\vskip .1 in 
\begin{table}[H]
	\setlength{\tabcolsep}{0.250cm}
	\renewcommand{\arraystretch}{1.50092}
	\setlength{\arrayrulewidth}{0.75pt}
	\centering
	\begin{tabular}{r|r|r|r|r|r|r|r}
		$p$	& $\tau(p)$ & $p$	& $\tau(p)$  &  $p$	& $\tau(p)$ &  $p$	& $\tau(p)$\\
		\hline
		1	& 1 & 17 &$-6905934$  &43  & $-17125708$&79& 38116845680\\
		\hline
		2	& $-24$ &19  &10661420  &47  &2687348496 &83&$ -29335099668$ \\
		\hline
		3	& 252 & 23 & 18643272 &59  &$-5189203740 $ &89&$-24992917110$\\
		\hline
		5	&4830  & 29 & 128406630 &61  & 6956478662&97&75013568546 \\
		\hline
		7	& $-16744$ & 31 & $-52843168$ & 67 &$-15481826884$&101  &81742959102\\
		\hline
		11	&534612  & 37 &$-182213314$  & 71 & 9791485272  &103&$-225755128648$\\
		\hline
		13	&$-577738$  & 41 &308120442  & 73 &1463791322 &107&90241258356 
	\end{tabular}
	\vskip .1 in		
	\caption{Small Values of the Tau Function}
	\label{table100-2}
\end{table}
%TTTTTTTTTTTTTTTTTTTTTTTTTTTTTTTTTTTTTTT

A large table of the values $\tau(n)$, $n\leq 16091$, is archived at \href{https://oeis.org/A000594/b000594.txt}{Table}. A demonstration of a technique for computing small values is rendered below.
%EEEEEEEEEEEEEEEEEEEEEEEEEEEEEEEEEEEEEEEEEEEEEEEEEEEEEEEEEE
\begin{exa}
	{\normalfont For $p=2$, the evaluation of the formula gives
		\begin{eqnarray}
			\tau(p)& =& -\frac{1}{2}(p)^2\sigma_{7}(p) + \frac{3}{2}(p)^2\sigma_{3}(p)+\frac{360}{p}S_3(3, p)\\[.2cm]
			& =& -\frac{1}{2}2^2\left( 2^7+1\right)  + \frac{3}{2}2^2\left(2^3+1 \right) +\frac{360}{2}S_3(3, p)\nonumber\\[.2cm]
			& =&-204+ 180(1)\nonumber\\
			& =&-24\nonumber		
		\end{eqnarray}
		since $S_3(3, p)=\sum_{1\leq k<p}k^3\sigma_3(k)\sigma_3(p-k)=1$.
	}
\end{exa}

%EEEEEEEEEEEEEEEEEEEEEEEEEEEEEEEEEEEEEEEEEEEEEEEEEEEEEEEEEE
\begin{exa}
	{\normalfont For $p=3$, the evaluation of the formula gives
		\begin{eqnarray}
			\tau(p)& =& -\frac{1}{2}\sigma_{7}(p) + \frac{3}{2}(p)^2\sigma_{3}(p)+\frac{360}{p}S_3(3, p)\\[.2cm]
			& =& -\frac{1}{2}3^3\left( 3^7+1\right)  + \frac{3}{2}3^2\left(3^3+1 \right) +\frac{360}{3}S_3(3, p)\nonumber\\[.2cm]
			& =&-9468+ 120(81)\nonumber\\
			& =&252\nonumber		
		\end{eqnarray}
		since $S_3(3, p)=\sum_{1\leq k<p}k^3\sigma_3(k)\sigma_3(p-k)=81$.
	}
\end{exa}

%EEEEEEEEEEEEEEEEEEEEEEEEEEEEEEEEEEEEEEEEEEEEEEEEEEEEEEEEEE
\begin{exa}
	{\normalfont For $p=5$, the evaluation of the formula gives
		\begin{eqnarray}
			\tau(p)& =& -\frac{1}{2}(p)^2\sigma_{7}(p) + \frac{3}{2}(p)^2\sigma_{3}(p)+\frac{360}{p}S_3(3, p)\\[.2cm]
			& =& -\frac{1}{2}5^2\left( 5^7+1\right)  + \frac{3}{2}5^2\left(5^3+1 \right) +\frac{360}{5}S_3(3, p)\nonumber\\[.2cm]
			& =& +120(13565)\nonumber\\
			& =&4830\nonumber		
		\end{eqnarray}
		since $S_3(3, p)=\sum_{1\leq k<p}k^3\sigma_3(k)\sigma_3(p-k)=13565$.
	}
\end{exa}
%SSSSSSSSSSSSSSSSSSSSSSSSSSSSSSSSSSSSSSSSSSSSSSSSSSSSSSSSSSSSSSSSSSSSSSSSSSSSSSSSSSS
%SSSSSSSSSSSSSSSSSSSSSSSSSSSSSSSSSSSSSSSSSSSSSSSSSSSSSSSSSSSSSSSSSSSSSSSSSSSSSSSSSSS
%SSSSSSSSSSSSSSSSSSSSSSSSSSSSSSSSSSSSSSSSSSSSSSSSSSSSSSSSSSSSSSSSSSSSSSSSSSSSSSSSSSS
%SSSSSSSSSSSSSSSSSSSSSSSSSSSSSSSSSSSSSSSSSSSSSSSSSSSSSSSSSSSSSSSSSSSSSSSSSSSSSSSSSSS
%SSSSSSSSSSSSSSSSSSSSSSSSSSSSSSSSSSSSSSSSSSSSSSSSSSSSSSSSSSSSSSSSSSSSSSSSSSSSSSSSSSS
%SSSSSSSSSSSSSSSSSSSSSSSSSSSSSSSSSSSSSSSSSSSSSSSSSSSSSSSSSSSSSSSSSSSSSSSSSSSSSSSSSSS
%SSSSSSSSSSSSSSSSSSSSSSSSSSSSSSSSSSSSSSSSSSSSSSSSSSSSSSSSSSSSSSSSSSSSSSSSSSSSSSSSSSS
\section{Bernoulli numbers} \label{S5599BN-BN}\hypertarget{S5599BN-BN}
A small table of Bernoulli numbers is included here as a reference. These rational numbers are the coefficients of the power series

\begin{equation}
	\frac{te^t}{e^t-1}=\sum_{n\geq0}B_n\frac{t^n}{n!}.
\end{equation}
%TTTTTTTTTTTTTTTTTTTTTTTTTTTTTTTTTTTTTTTTTTTTTTTTTTTTTTTTTTTTTTTTTTTTTTTTTT
\vskip .1in
\begin{table}[H]
	\setlength{\tabcolsep}{0.75cm}
	\renewcommand{\arraystretch}{1.750092}
	\setlength{\arrayrulewidth}{0.750pt}
	\centering
	
	\begin{tabular}{r|r|r|r|r|r}
		$n$&$B_{n}$&$n$&$B_{n}$&$n$&$B_{n}$\\
		\hline
		$0$ &$1$&$8$&$-\frac{1}{30}$&$18$&$-\frac{43867}{798}$\\
		\hline
		$1$ &$-\frac{1}{2}$&$10$&$-\frac{5}{66}$&$20$&$-\frac{174611}{330}$\\
		\hline
		$2$ &$\frac{1}{6}$&$12$&$-\frac{691}{2370}$&$22$&$\frac{854513}{138}$\\
		\hline
		$4$ &$-\frac{1}{30}$&$14$&$\frac{7}{6}$&$24$&$-\frac{236364091}{2730}$\\
		\hline
		$6$ &$\frac{1}{42}$&$16$&$-\frac{3617}{510}$&$26$&$\frac{8553103}{6}$\\
	\end{tabular}
	\caption{List of Bernoulli Numbers $B_{n}$} \label{t901}
\end{table}
\vskip .1in
%TTTTTTTTTTTTTTTTTTTTTTTTTTTTTTTTTTTTTTTTTTTTTTTTTTTTTTTTTTTTTTTTTTTTTTTTTT
A large and more detailed table appears in {\color{red}\cite[Table 24.2.3]{DLMF}}.
%\newpage

%PPPPPPPPPPPPPPPPPPPPPPPPPPPPPPPPPPPPPPPPPPPPPPPPPPPPPPPPPPPPPPPPPPPPPPPPPPPPPPPPPPPPPPPPP
%PPPPPPPPPPPPPPPPPPPPPPPPPPPPPPPPPPPPPPPPPPPPPPPPPPPPPPPPPPPPPPPPPPPPPPPPPPPPPPPPPPPPPPPPP
%PPPPPPPPPPPPPPPPPPPPPPPPPPPPPPPPPPPPPPPPPPPPPPPPPPPPPPPPPPPPPPPPPPPPPPPPPPPPPPPPPPPPPPPPP
%PPPPPPPPPPPPPPPPPPPPPPPPPPPPPPPPPPPPPPPPPPPPPPPPPPPPPPPPPPPPPPPPPPPPPPPPPPPPPPPPPPPPPPPPP
%PPPPPPPPPPPPPPPPPPPPPPPPPPPPPPPPPPPPPPPPPPPPPPPPPPPPPPPPPPPPPPPPPPPPPPPPPPPPPPPPPPPPPPPPP
%PPPPPPPPPPPPPPPPPPPPPPPPPPPPPPPPPPPPPPPPPPPPPPPPPPPPPPPPPPPPPPPPPPPPPPPPPPPPPPPPPPPPPPPPP
%PPPPPPPPPPPPPPPPPPPPPPPPPPPPPPPPPPPPPPPPPPPPPPPPPPPPPPPPPPPPPPPPPPPPPPPPPPPPPPPPPPPPPPPPP
%PPPPPPPPPPPPPPPPPPPPPPPPPPPPPPPPPPPPPPPPPPPPPPPPPPPPPPPPPPPPPPPPPPPPPPPPPPPPPPPPPPPPPPPPP
\section{Problems}\label{p9988}

\subsection{Techniques for Computing $\tau(n)$}
%EEEEEEEEEEEEEEEEEEEEEEEEEEEEEEEEEEEEEEEEEEEEEEEEEEEEEEEEEEEEEEEEEEEEEEEEEEEEEEEE
\begin{exe} {\normalfont Let $r_{24}(n)\geq0$ be the number of representations of $n\geq1$ as a sum of 24 squares. 
\begin{enumerate}
\item[(a)] Use $r_{24}(1)=2\cdot 24$ to compute $\tau(1)$.
\item[(b)] Use $r_{24}(2)=2^2\cdot 23$ to compute $\tau(2)$.
\item[(c)] Use $r_{24}(1)=2^3\cdot 22$ to compute $\tau(3)$.
\item[(d)] Explain whether or not a known value $r_{24}(n)=m$ is practical for computing $\tau(n)$ with $n$ large.
\end{enumerate} 
	}
\end{exe}
\vskip .15in 
%EEEEEEEEEEEEEEEEEEEEEEEEEEEEEEEEEEEEEEEEEEEEEEEEEEEEEEEEEEEEEEEEEEEEEEEEEEEEEEEE
%EEEEEEEEEEEEEEEEEEEEEEEEEEEEEEEEEEEEEEEEEEEEEEEEEEEEEEEEEEEEEEEEEEEEEEEEEEEEEEEE
\begin{exe} {\normalfont Let $\Delta(x)=\sum_{n\geq1}\tau(n)x^{n-1}=x\prod_{n\geq}\left(1-x^n \right)^{24}$. The value $\tau(2)=-24$ be computed directly by expanding the product and matching coefficients. 
\begin{enumerate}
	\item[(a)] Compute $\tau(1)=1$ by matching coefficients.
	\item[(b)] Compute $\tau(3)$ by matching coefficients.
	\item[(c)] Show that $\tau(n)\ne-1$ by matching coefficients (if posssible). 
	\item[(d)] Explain whether or not this method is practical for computing $\tau(n)$ with $n$ large.
\end{enumerate} 
	}
\end{exe}
\vskip .15in 
%EEEEEEEEEEEEEEEEEEEEEEEEEEEEEEEEEEEEEEEEEEEEEEEEEEEEEEEEEEEEEEEEEEEEEEEEEEEEEEEE
\subsection{One-to-One Properties of the Fourier Coefficients $\lambda(n)$}
\begin{exe} {\normalfont Prove or disprove that the Ramanujan coefficient $\tau:\N \longrightarrow \Z$ is injective, (one-to-one).
	}
\end{exe}
\vskip .15in 
%EEEEEEEEEEEEEEEEEEEEEEEEEEEEEEEEEEEEEEEEEEEEEEEEEEEEEEEEEEEEEEEEEEEEEEEEEEEEEEEE
\begin{exe} {\normalfont Let $k=12, 14, 16, 18, 20, 22, 26$ be the weight, and let $N=1$ be the level of the Fourier coefficient $\lambda(n)$. Prove or disprove that the map $\lambda:\N \longrightarrow \Z$ is injective, (one-to-one).
	}
\end{exe}
\vskip .15in 
%EEEEEEEEEEEEEEEEEEEEEEEEEEEEEEEEEEEEEEEEEEEEEEEEEEEEEEEEEEEEEEEEEEEEEEEEEEEEEEEE
%EEEEEEEEEEEEEEEEEEEEEEEEEEEEEEEEEEEEEEEEEEEEEEEEEEEEEEEEEEEEEEEEEEEEEEEEEEEEEEEE
%EEEEEEEEEEEEEEEEEEEEEEEEEEEEEEEEEEEEEEEEEEEEEEEEEEEEEEEEEEEEEEEEEEEEEEEEEEEEEEEE
\subsection{Vanishing Fourier Coefficients $\lambda(n)$}
\begin{exe} {\normalfont Let $k=12, 14, 16, 18, 20, 22, 26$ be the weight, and let $N=1$ be the level of the Fourier coefficient $\lambda(n)$. Prove that $\lambda(n)\ne 0$ for $n \geq 1$.
	}
\end{exe}

\begin{exe} {\normalfont Let $k\geq 1$ be the weight, and let $N\geq1$ be the level of the Fourier coefficient $\lambda(n)$. Prove or disprove that the sign of $\lambda(n)$ is a random variable. Is there a deterministic algorithm to compute it?
	}
\end{exe}

\begin{exe} {\normalfont Let $k\geq 1$ be the weight, and let $N\geq1$ be the level of the Fourier coefficient $\lambda(n)$. Prove or disprove that the number of prime values
		$$\pi_{\lambda}(x,k,N)=\#\{\lambda(n)\leq x\}=c_{\lambda}(k,N)\log \log x+o(\log \log x,$$
		where $c_{\lambda}(k,N)\geq 0$ is a constant. The heuristic argument for $\lambda(n)=\tau(n)$ appears in \cite{LR2013}.
	}
\end{exe}
\begin{exe} {\normalfont Let $k\geq 1$ be the weight, and let $N\geq1$ be the level of the Fourier coefficient $\lambda(n)$. Estimate the maximal prime gap
		$$\left |\lambda(p^m)\right |-\left |\lambda(p^n)\right |,$$
		where $p\in \tP=\{2,3,5, \ldots\}$, and $m,n \in \N=\{1,2,3, \ldots\}$. Does it depends on the rational prime gap $p_{n+1}-p_n$? 
	}
\end{exe}

%EEEEEEEEEEEEEEEEEEEEEEEEEEEEEEEEEEEEEEEEEEEEEEEEEEEEEEEEEEEEEEEEEEEEEEEEEEEEEEEE
\begin{exe} {\normalfont Let $k\geq 1$ be the weight, and let $N\geq1$ be the level of the Fourier coefficient $\lambda(n)$, and no complex multiplication. Does $\lambda(p)\ne0$ for any prime $p\in \tP=\{2,3,5, \ldots\}$ if and only if $k\geq 4$? 
	}
\end{exe}
\vskip .15in 
%EEEEEEEEEEEEEEEEEEEEEEEEEEEEEEEEEEEEEEEEEEEEEEEEEEEEEEEEEEEEEEEEEEEEEEEEEEEEEEEE

%BBBBBBBBBBBBBBBBBBBBBBBBBBBBBBBBBBBBBBBBBBBBBBBBBBBBBBBBBBBBBBBBBBBBBBBBB
%BBBBBBBBBBBBBBBBBBBBBBBBBBBBBBBBBBBBBBBBBBBBBBBBBBBBBBBBBBBBBBBBBBBBBBBBB
%BBBBBBBBBBBBBBBBBBBBBBBBBBBBBBBBBBBBBBBBBBBBBBBBBBBBBBBBBBBBBBBBBBBBBBBBB
%BBBBBBBBBBBBBBBBBBBBBBBBBBBBBBBBBBBBBBBBBBBBBBBBBBBBBBBBBBBBBBBBBBBBBBBBB
%BBBBBBBBBBBBBBBBBBBBBBBBBBBBBBBBBBBBBBBBBBBBBBBBBBBBBBBBBBBBBBBBBBBBBBBBB

\currfilename.\\

\end{document}